\documentclass[10pt,final]{article}
\usepackage[a4paper]{geometry}
\usepackage[utf8]{inputenc}
\usepackage{amsmath}
\usepackage{amssymb}
\usepackage{amsthm}
\usepackage{amsopn}
\usepackage{mathrsfs}
\usepackage{exscale}
\usepackage{enumitem}
\usepackage{color}
\usepackage{cite}
\usepackage{subfigure}
\usepackage{pgfplots}
\usepackage{tikz}
\pgfplotsset{compat=newest}
\newlength\fheight \newlength\fwidth

\usepackage{mathabx}
\usepackage{todonotes}
\usepackage{xcolor}
\usepackage{csquotes}
\usepackage[verbose, colorlinks=true, pdfstartview=FitV, linkcolor=blue!60!black, citecolor=blue!60!black, urlcolor=blue!60!black]{hyperref}
\usepackage[capitalise,noabbrev]{cleveref}
\usepackage{enumitem}
\usepackage{multirow}

\newcommand{\X}{\mathcal X}
\newcommand{\Y}{\mathcal Y}

\newcommand{\R}{\mathbb R}    
\newcommand{\N}{\mathbb{N}}   

\newcommand*\diff{\mathop{}\!\mathrm{d}}

\DeclareMathOperator*{\argmin}{arg\,min}

\DeclareMathOperator*{\trace}{tr}

\newcommand{\data}{Y}   
\newcommand{\sig}{f}   
\newcommand{\sigdag}{\sig^\dagger}  
\newcommand{\err}{Z}  
\newcommand{\flt}{q}    
\newcommand{\fop}{T}     
\newcommand{\bip}{\textsc{b}} 
\newcommand{\vap}{\textsc{v}} 
\newcommand{\vop}{\textsc{u}} 
\newcommand{\nop}{\textsc{n}} 
\newcommand{\risk}{\mathcal{R}} 
\newcommand{\thd}{\eta} 
\newcommand{\opt}{\mathrm{opt}} 

\renewcommand{\hat}{\widehat} 

\theoremstyle{plain}
\newtheorem{theorem}{Theorem}[section]
\newtheorem{lemma}[theorem]{Lemma}
\newtheorem{prop}[theorem]{Proposition}

\theoremstyle{definition}
\newtheorem{assumption}{Assumption}
\newtheorem{definition}{Definition}

\newtheorem{remarkx}[theorem]{Remark}
\newenvironment{remark}
{\pushQED{\qed}\remarkx} 
{\popQED\endremarkx}

\newcommand{\abs}[1]{\left|#1\right|}
\newcommand{\norm}[1]{\left\Vert#1\right\Vert}
\newcommand{\snorm}[1]{\Vert#1\Vert}
\newcommand{\set}[2]{\left\{#1\, :\, #2\right\}}
\newcommand{\ExSign}{\mathbb E}
\newcommand{\E}[1]{{\ExSign}\left[ #1 \right]}
\newcommand{\sE}[1]{{\ExSign}[ #1 ]}
\newcommand{\PrSign}{\mathbb P}
\newcommand{\Prob}[1]{{\PrSign}\left\{#1 \right\}}

\newcommand{\svar}[1]{{\mathrm{Var}}( #1 )}

\newcommand{\eref}[1]{\eqref{#1}}

\newcommand{\rev}{}

\author{Housen Li and Frank Werner}

\date{\today}

\begin{document}

\title{Adaptive minimax optimality in statistical inverse problems via SOLIT --- Sharp Optimal Lepski\u{\i}-Inspired Tuning}

\author{Housen Li${}^{1}$,  and Frank Werner${}^{2}$\\
\vspace{0.1cm}\\
{\small${}^{1}$Institute for Mathematical Stochastics, University of G\"ottingen, Germany}\\
{\small${}^{2}$Institute of Mathematics, University of Würzburg, Germany}}

\maketitle

\begin{abstract}
		We consider statistical linear inverse problems in separable Hilbert spaces and filter-based reconstruction methods of the form $\hat f_\alpha = q_\alpha \left(T^*T\right)T^*Y$, where $Y$ is the available data, $T$ the forward operator, $\left(q_\alpha\right)_{\alpha  \in \mathcal A}$ an ordered filter, and $\alpha > 0$ a regularization parameter. Whenever such a method is used in practice, $\alpha$ has to be appropriately chosen. Typically, the aim is to find or at least approximate the best possible $\alpha$ in the sense that mean squared error (MSE) $\mathbb E [\Vert \hat f_\alpha - f^\dagger\Vert^2]$ w.r.t.~the true solution $f^\dagger$ is minimized. In this paper, we introduce the Sharp Optimal Lepski\u{\i}-Inspired Tuning (SOLIT) method, which yields an a posteriori parameter choice rule ensuring adaptive minimax rates of convergence. It depends only on $Y$ and the noise level $\sigma$ as well as the operator $T$ and the filter $\left(q_\alpha\right)_{\alpha  \in \mathcal A}$ and does not require any problem-dependent tuning of further parameters. We prove an oracle inequality for the corresponding MSE in a general setting and derive the rates of convergence in different scenarios. By a careful analysis we show that no other a posteriori parameter choice rule can yield a better performance in terms of {\rev the order of} the convergence rate of the MSE. In particular, our results reveal that the typical understanding of Lepski\u\i-type methods in inverse problems leading to a loss of a log factor is wrong. In addition, the empirical performance of SOLIT is examined in simulations. 
\end{abstract}

\section{Introduction}
	
	In this paper, we consider statistical linear inverse problems of the form 
	\begin{equation}\label{e:model}
		\data = \fop \sigdag  + \sigma \err
	\end{equation}
	with a linear and injective Hilbert--Schmidt operator $\fop : \X \to \Y$ mapping between separable Hilbert spaces $\X$ and $\Y$, and the standard Gaussian Hilbert space process $\err: \Y \to \mathbf L^2 \left(\Omega, \mathcal A, \mathcal P\right)$. Here and in what follows, we denote by $\sigdag \in \X$ the true (but unknown) solution, by $\sigma > 0$ the noise level, and by $\data$ the available measurements. Note that $\err$ in \eref{e:model} can be identified with an element in $\Y$ only if $\dim(\Y) < \infty$. Otherwise, we have $\data \notin \Y$ a.s., which implies that the model \eref{e:model} has to be understood in a weak sense, i.e.,  for all $y \in \Y$ we have access to the real-valued random variable
	\[
	\left\langle \data, y\right\rangle := \left\langle \fop \sigdag,y\right\rangle + \sigma \err(y).
	\]
	Statistical inverse problems of form \eref{e:model} are a widely investigated field as they arise in applications spanning from astronomy over medical imaging to engineering, see, e.g., \cite{hw16,bbdv09,os86}. The assumption of Gaussian white noise as in \eref{e:model} is most common and can be justified by either central limit theorems or asymptotic equivalence statements \cite{rsh18}. 
	
	To reconstruct the quantity of interest $\sigdag$ from the available data $\data$, we consider the filter-based regularization
	\begin{equation}\label{e:ualhat}
		\hat \sig_\alpha = \flt_\alpha \left(\fop^*\fop\right) \fop^*\data,
	\end{equation}
	where $\left(q_\alpha\right)_{\alpha > 0}$ is an ordered filter (cf.\ \cref{d:of} below). Spectral regularization schemes \eref{e:ualhat} based on filters were introduced by Bakushinski \cite{b67} and yield a variety of regularized estimators for the desired quantity $\sigdag$. They have also received substantial attention in the literature (see, e.g., the monograph \cite{ehn96} for a deterministic treatment and \cite{BHMR07,MP06,LiWe20} for a stochastic investigation). 
	
	An important property of the estimators $\hat \sig_\alpha$ under investigation is their behavior as the noise level $\sigma$ tends to $0$. Due to the famous result by Schock \cite{s85}, this behavior, if measured in the rate of convergence of the mean squared error (MSE) $\mathbb E \bigl[\bigl\Vert \hat \sig_\alpha - \sigdag\bigr\Vert_{\X}^2\bigr]$, can be arbitrarily slow. However, under additional smoothness conditions on $\sigdag$, it is in many situations possible to derive convergence rates
	\[
	\mathbb E \left[\bigl\Vert \hat \sig_{\alpha^*} - \sigdag\bigr\Vert_{\X}^2\right] = \mathcal O \left(\psi \left(\sigma\right)\right)
	\]
	with a monotonically increasing and continuous rate function $\psi : \R_{\ge 0} \to \R_{\ge 0}$, $\psi(0) = 0$, for the ideal parameter choice {\rev $\alpha^\opt = \alpha^\opt \left(\sigma\right)$} given by
	\begin{equation*}
		\alpha^\opt := \argmin_{\alpha > 0} \mathbb E \left[\bigl\Vert \hat \sig_\alpha - \sigdag\bigr\Vert_{\X}\right].
	\end{equation*}
    {\rev We stress that the focus of this paper is on the order of convergence rate, namely, $\psi(\sigma)$, rather than the multiplying constant hidden in Landau's big O notation. Thus, for brevity, the optimality statement throughout is meant for the order of convergence rate, not the multiplying constant, without  explicit declaration.}
	Under reasonable assumptions on the filter $\left(q_\alpha\right)_{\alpha > 0}$, the rates of convergence can actually be shown to be minimax optimal, i.e., best possible under all estimators uniformly over a smoothness class (cf.\ \cref{s:moa} for details). This is a strong indicator that the chosen family of estimators will perform well in practice, but it is unfortunately useless for applications since the oracle parameter choice $\alpha^\opt$ cannot be implemented (as $\sigdag$ is unknown). Therefore, one is actually interested {\rev in} practically implementable (i.e., not depending on unknown quantities such as $\sigdag$ or its smoothness) choices of $\alpha$ that mimic the minimax convergence behavior. 
	
	In this paper, we investigate a specific a posteriori parameter choice rule, which we call \textbf{S}harp \textbf{O}ptimal \textbf{L}epski\u\i-\textbf{I}nspired \textbf{T}uning (SOLIT). In terms of the estimator \eref{e:ualhat}, the corresponding parameter $\hat \alpha$ is constructed as follows. We first choose a suitable finite set of candidates $\left\{\alpha_0, ..., \alpha_{m_{\max}}\right\}$, ordered increasingly with respect to variance (or model complexity), and then set $\hat\alpha:= \alpha_{\hat m}$ with
	$$
	\hat m := \min\set{0 \leq m_1 \leq m_{\max}}{\max_{m_{\max}\geq m_2 > m_1} \bigl(\snorm{{\hat{\sig}_{\alpha_{m_1}} - \hat{\sig}_{\alpha_{m_2}}}} - \thd_{m_1,m_2}\bigr) \le 0},
	$$
	where $\thd_{m_1,m_2} \ge 0$ are explicitly computable critical values depending only on the noise level~$\sigma$, the filter $\flt_{\alpha}$, and the forward operator $\fop$ (see \cref{s:solit} for the formal definition and \cref{sec:z} for computational issues). It is derived from a more general principle proposed in \cite{SpWi19}, and our analysis exploits and extends results from that paper.
	
	The SOLIT rule $\hat\alpha \equiv \alpha_{\hat m}$ can be seen as a variant of Lepski\u\i-type balancing principle. The Lepski\u\i-type balancing principles are well-known in inverse problems and typically of the form $\alpha_{\mathrm{LEP}} = \alpha_{m_{\mathrm{LEP}}}$ with
	\begin{equation}\label{eq:lepskij}
	\hspace{-2cm}m_{\mathrm{LEP}}:= \min\set{{\rev 0} \leq m_1 \leq m_{\max}}{\max_{m_{\max} \ge m_2 > m_1}\left(\snorm{\hat{\sig}_{\alpha_{m_1}} - \hat{\sig}_{\alpha_{m_2}}}  - 4\kappa \mu_{m_2} \right) \leq 0},
	\end{equation}
	where $\mu_{k}:=\sigma\sqrt{\mathrm{trace}\bigl(\flt_{\alpha_k}\left(\fop^*\fop\right)^2\fop^*\fop\bigr)}$ and $\kappa \geq 1$ is a tuning parameter, see, e.g., \cite{mp03,bh05,m06,MP06,wh12,w18}. Parameter choice rules of the form $\alpha_{\mathrm{LEP}}$ have originally been introduced in \cite{l90}, and the usage in inverse problems is summarized in the overview paper \cite{m06}. It has been shown in \cite{bp05,MP06} that $\alpha_{\mathrm{LEP}}$ leads to an oracle inequality which allows to compare the MSE under $\alpha_{\mathrm{LEP}}$, i.e., $\mathbb E \left[\bigl\Vert \hat \sig_{\alpha_{\mathrm{LEP}}} - \sigdag\bigr\Vert_{\X}^2\right]$, with the minimal MSE $\mathbb E \left[\bigl\Vert \hat \sig_{\alpha^\opt} - \sigdag\bigr\Vert_{\X}^2\right]$ with high probability up to a $\sqrt{m_{\max}}\exp\left(-C\kappa^2\right)$ term. From this one can deduce that a reasonable choice of $\kappa$ (depending on the minimax rate of convergence, see \cref{s:moa} below, and hence on the ill-posedness of $\fop$) leads to minimax rates of convergence up to a logarithmic factor.
	
	Opposed to this, we will show that the SOLIT rule $\hat\alpha \equiv \alpha_{\hat m}$ yields the best possible rate (precisely the \textit{adaptive} minimax rate, see also \cref{s:moa} below) of convergence without loss of logarithmic factors and without the necessity to adjust any tuning parameters to the  ill-posedness of $\fop$ over a wide range of smoothness classes for $\sigdag$. One reason behind this is that the SOLIT $\hat \alpha$ employs a subtle upper bound $\thd_{m_1,m_2}$ depending on both $m_1$ and $m_2$, while the Lepski\u\i-type balancing principle $\alpha_{\mathrm{LEP}}$ utilises an upper bound $\mu_{m_2}$ depending just on $m_2$. Also importantly, SOLIT requires a special choice of the candidate parameters $\alpha_0, \ldots, \alpha_{m_{\max}}$, see \cref{ss:alpha} below. 
	
	The outline of this paper is as follows. In \cref{s:moa} we discuss the questions of minimax optimality and adaptation in inverse problems and provide the corresponding rates of convergence in a wide class of problems. This also serves as an overview or a review of already known results while adding some new aspects. \cref{s:solit} is then devoted to filter-based regularization and a detailed derivation of our parameter choice SOLIT. In this section, we also state a generic oracle inequality for the risk based on the SOLIT. The final result on rates of convergence including all necessary assumptions and computations is then presented in \cref{s:rates}. In \cref{sec:implementation} we discuss how the SOLIT can be implemented and provide details on the computation of the critical values $\thd_{m_1,m_2}$ involved in $\hat\alpha$. Furthermore we present various numerical simulations supporting our analysis, before we end this paper with some conclusions in \cref{sec:conclusions}.
	
	\section{Minimax optimality and adaptation}\label{s:moa}
	
	Before we define and analyze the SOLIT, let us recall the minimax paradigm in statistical inverse problems. Note that the forward operator $\fop$ in \eref{e:model} is assumed to be compact, injective and Hilbert--Schmidt. Thus, there is an eigendecomposition of $\fop^*\fop$ with strictly positive eigenvalues $\{\lambda_k\}_{k\in\N}$ ordered decreasingly and corresponding eigenfunctions $\{e_k\}_{k \in \N}$, such that 
	\begin{equation}\label{e:eigdc}
		\fop^*\fop e_k = \lambda_k e_k, \quad \lambda_1\ge \lambda_2 \ge \cdots > 0\quad\mbox{and} \quad \sum_{k=1}^\infty\lambda_k < \infty.
	\end{equation}
	We stress that this decomposition is only used in our theoretical analysis and is not necessarily required for the construction of our estimator or the SOLIT parameter choice, {\rev see \cref{sec:implementation}.}
	
	The degree of ill-posedness of the model \eref{e:model} is often characterized by the decay rate $\lambda_k \to 0$ as $k \to \infty$. Most commonly assumed is a polynomial rate of decay 
	\begin{equation}\label{e:eigop}
		c_{\min}k^{-a} \le \lambda_k \le c_{\max} k^{-a} \quad \mbox{for all } k \in \N
	\end{equation}
	with $0 < c_{\min} \leq c_{\max}$ and $a > 0$. This is often referred to as \emph{mild ill-posedness.} Applications where this occurs include medical X-ray tomography (with $\fop$ the $d$-dimensional Radon transform, $d\in \N_{\ge 2}$, which satisfies \eref{e:eigop} with $a = d-1$, see \cite{Nat01}), parameter identification problems in partial differential equations (see, e.g., \cite{Isa17}), and density estimation in statistics (see, e.g., \cite{Tsy09}), to name a few. Also widely accepted are models with an exponential decay of the singular values
	\begin{equation}\label{e:eigopb}
		c_{\min}k^{-a} \exp(-b k^{\vartheta}) \le \lambda_k \le c_{\max} k^{-a} \exp(-b k^{\vartheta})\quad \mbox{for all } k \in \N
	\end{equation}
	with $a \in \R$ and $b, \vartheta > 0$, which is typically called \emph{severe ill-posedness}. Some examples include inverse heat equation problems \cite{EnRu95} as well as image deblurring and light microscopy (where $\fop$ is the convolution operator with, e.g., a Gaussian kernel) \cite{KoTs93}. In most of this paper, we will focus on mild ill-posedness \eref{e:eigop}, but in \cref{s:sip} we will also study the severely ill-posed case \eref{e:eigopb}.
	
	As mentioned in the Introduction, nothing can be said without further assumptions on the unknown $\sigdag $. Typically, one poses spectral source conditions of the form
	\begin{equation}\label{e:scd}
		\sigdag \in \mathcal W_\varphi(\rho) := \set{\sig \in \X }{ \sig = \varphi\left(\fop^*\fop\right)g\quad \mbox{for }~\norm{g}_\X \leq \rho}
	\end{equation}
	with an \emph{index} function $\varphi : \R_{\ge 0} \to \R_{\ge 0}$ (i.e., monotonically increasing, continuous and $\varphi(0) = 0$). With the eigendecomposition of $\fop^*\fop$ in \eref{e:eigdc}, we can rewrite the source condition \eref{e:scd} equivalently as an ellipsoid 
	$$
	\mathcal{W}_\varphi(\rho) = \set{ f = \sum_{k = 1}^\infty f_k e_k}{\sum_{k=1}^\infty w_k f_k^2 \le \rho^2} \subseteq \X,
	$$
	where $w_k = \varphi(\lambda_k)^{-2}$ for $k \in \N$. For several operators $\fop$, spectral source conditions with specific index functions $\varphi$ can be interpreted as classical smoothness conditions in terms of Sobolev spaces. Note that for every $\sigdag$ there exist an index function $\varphi$ and a parameter $\rho > 0$ such that $\sigdag \in \mathcal W_{\varphi} \left(\rho\right)$ (see \cite{mh08}), but the decay behavior of $\varphi$ at $0$ can be arbitrarily slow. 
	
	We are particularly interested in index functions $\varphi$ of form
	\begin{equation}\label{e:indfun}
		\varphi(x) = \varphi_{\nu,\tau}(x) := x^\nu(- \log x)^{-\tau}\qquad \mbox{for } x \ge 0,
	\end{equation}
	where either $\nu > 0$ and $\tau \in \R$ (referred to as \emph{polynomial smoothness}), or $\nu=0$ and $\tau > 0$ (referred to as \emph{logarithmic smoothness}). To ensure well-definedness of $\varphi_{\nu, \tau}\left(\fop^*\fop\right)$, we assume that $\lambda_1 \leq \exp(-1)$, without loss of generality.  
	
	Given an index function $\varphi$ and a parameter $\rho > 0$, one can now ask for the minimax risk of the MSE, denoted by
	\begin{equation}\label{e:or}
		\psi \left(\sigma\right) = \psi\left(\sigma, \varphi,\rho,\fop\right) := \inf_{\hat \sig~\mathrm{estimator}} \sup_{\sigdag \in \mathcal W_{\varphi}\left(\rho\right)} \mathbb E \left[\bigl\Vert \hat \sig - \sigdag\bigr\Vert_{\X}^2\right].
	\end{equation}
    {\rev Here and below (e.g., in \eref{e:orpa}) an estimator is any measurable mapping from $\Y^*$ to $\X$, with typical Borel $\sigma$-algebras, where $\Y^*$ denotes the space of linear mappings from $\Y$ to $\mathbf L^2 \left(\Omega, \mathcal A, \mathcal P\right)$.} 
	To derive the minimax risk \eref{e:or}, it turns out that spectral cutoff --- corresponding to the filter-based regularization \eref{e:ualhat} with $\flt_{\alpha}(\lambda) := \lambda^{-1}\mathbf{1}\{\lambda \ge \alpha\}$ --- is particularly helpful. It has been shown, e.g., in \cite{Erm89,Koo93,MaRu96}, that, for known smoothness parameters $\nu$ and $\tau$, the spectral cutoff estimator achieves the minimax rate of convergence (this is the same rate as the minimax risk in \eref{e:or} as $\sigma \to 0$) with the index function $\varphi_{\nu, \tau}$ in \eref{e:indfun}. Furthermore, a modified spectral cutoff \cite{Pin80}, often known as Pinsker's filter, is shown to achieve asymptotically the exact minimax risk (i.e., front constants are also optimal) in many cases, e.g., when $\sup_{k \in \N} \lambda_k/\lambda_{k+1} < \infty$, and the radius $\rho$ is also known, see \cite{GoKh99}.
	
	However, note that the minimax risk \eref{e:or} allows the estimator to have knowledge of $\varphi$, which is typically unknown in practice. A more realistic setting is to suppose that $\varphi \in \Phi$ with a collection of index functions $\Phi$. For example, it might be a priori known that $\sigdag$ is of polynomial smoothness, which corresponds to
	\begin{equation}\label{e:cps}
		\Phi_{\nu_0} = \set{\varphi = \varphi_{\nu,\tau}}{0 < \nu \le \nu_0,\,\tau \in \R},   
	\end{equation}
	or that a $\sigdag$ is of logarithmic smoothness, i.e.,
	\begin{equation}\label{e:cls}
		\Phi_{\tau_0} = \set{\varphi = \varphi_{\nu,\tau}}{ \nu = 0,\, 0 <\tau \le \tau_0},    
	\end{equation}
	where $\nu_0,\tau_0 > 0$ are known. In this case, one is interested to \emph{adapt to the unknown smoothness over $\Phi$}, 
 {\rev and wants to achieve the minimal price of adaptation over $\Phi$ with respect to the MSE, defined as 
    \begin{equation}\label{e:orpa}
    \varpi (\sigma) = \varpi (\sigma, \Phi, \rho, \fop) := \inf_{\hat \sig~\mathrm{estimator}}\; \sup_{\varphi \in \Phi} \; \sup_{\sigdag \in \mathcal W_{\varphi}\left(\rho\right)}\frac{ \mathbb E \left[\bigl\Vert \hat \sig - \sigdag\bigr\Vert_{\X}^2\right]}{\psi(\sigma, \varphi, \rho, \fop)},
    \end{equation}
    with $\psi(\sigma, \varphi, \rho, \fop)$ defined in \eref{e:or}. 
    For a fixed $\varphi \in \Phi$, we refer to 
        \begin{equation}\label{e:oar}
\psi_*(\sigma) = \psi_*(\sigma, \varphi, \Phi, \rho, \fop):= \varpi (\sigma, \Phi, \rho, \fop) \psi(\sigma, \varphi, \rho, \fop)
\end{equation}
as the adaptive minimax risk of the MSE over $\Phi$.
Clearly, it holds that the adaptive minimax risk $\psi_*(\sigma, \varphi, \Phi, \rho, \fop)$ is always no better (i.e., no faster decay as the argument tends to $0$) than the minimax risk $\psi(\sigma, \varphi, \rho, \fop)$, since $\varpi(\sigma) \ge 1$.} There are situations that both risks coincide up to a constant factor and also situations that both differ by a logarithmic factor depending on $\sigma$ even when $\Phi$ contains only two index functions. 
	
	Determining the adaptive minimax risk in \eref{e:oar} {\rev or equivalently the minimal price of adaptation in \eref{e:orpa}} is more delicate. In case of logarithmic smoothness (i.e., $\Phi_{\tau_0}$), the convergence rate of the adaptive minimax risk \eref{e:oar} coincides with that of the (non-adaptive) minimax risk \eref{e:or} even for $\Phi := \bigcup_{\tau_0 > 0}\Phi_{\tau_0}$.
	Interestingly, the order of optimal adaptation rates can be attained by, e.g., the spectral cutoff with regularization parameter $\alpha \asymp \sigma$, which depends solely on the noise level, and thus requires no knowledge of the smoothness parameter $\tau$ and the radius $\rho$ (see \cref{p:nlr} in \ref{s:lbar} for general ordered filters). 
	In case of polynomial smoothness (i.e., $\Phi_{\nu_0}$), the relation between the minimax risk and the adaptive minimax risk is subtle and depends on the degree of ill-posedness. 
	For mildly ill-posed problems (which we mainly consider), the optimal adaptation rates can be equal to the optimal rates, as it shown by \cite{CaTs02} using a spectral cutoff variant together with a penalized blockwise Stein’s rule, see also \cite{LiWe20}. For severely ill-posed problems (i.e., eigenvalues of $\fop^*\fop$ decaying at an exponential rate; cf.\ \eref{e:eigopb} and \cref{s:sip}), the optimal adaptation rates can differ from the optimal rates by a logarithmic factor. This has, e.g., been shown \cite{Tsy00}, and we will discuss this in detail in \cref{s:sip}.  The difference to the minimax risk is exactly the price that one has to pay for the missing information about the underlying smoothness of the truth. 
	
	One of the main results of this paper is that the filters of form \eref{e:ualhat} with the SOLIT rule are able to attain the sharp order of optimal adaptation rates for different combinations of smoothness and mild ill-posedness without manually adjusting the tuning parameters. In particular, this shows that the typical understanding in inverse problems that one has to lose a log-factor (see, e.g., Theorem~1 in \cite{MP06}) when applying Lepski\u\i-type methods is \emph{not} correct.
		
	\section{Filter-based regularization and SOLIT}\label{s:solit}
	
	Before we introduce the SOLIT rule, let us briefly set up the notation for filter-based regularization of \eref{e:model}. 
	
	\subsection{Filter-based regularization}
	
	In this paper, we consider the estimator $\hat\sig_\alpha$ in form of an \emph{ordered filter} $\flt_\alpha(\cdot)$, i.e.,
$$
	\hat\sig_\alpha = \flt_\alpha(\fop^*\fop) \fop^* \data,
$$
where $\alpha$ is a tuning parameter (i.e., regularization parameter).  
	
	\subsubsection{Ordered filters}\label{sss:of}

	\begin{definition}[Ordered filters]\label{d:of}
		The function $\flt_{\alpha}(\cdot): [0,\; \norm{\fop^*\fop}] \to \R_{\ge 0}$, indexed by $\alpha \in \mathcal{A}$, is called an \emph{ordered filter}, if it satisfies:
		\begin{enumerate}
			\item\label{flt:upb}
			For all $x \in [0,\; \norm{\fop^*\fop}]$ and for all $\alpha \in \mathcal{A}$, 
			$$
			\alpha \,{\flt_{\alpha}(x)} \le C_{\flt}'\quad \mbox{ and } \quad   x\, {\flt_{\alpha}(x)} \le C_{\flt}''\equiv 1.
			$$
			\item\label{flt:ord}
			It is \emph{ordered} with respect to \enquote{variance}, i.e.,
			$$
			\alpha \ge \tilde{\alpha} \quad \Rightarrow\quad \flt_{\alpha}(x) \le \flt_{\tilde{\alpha}}(x)\; \mbox{ for all } x \in [0,\; \norm{\fop^*\fop}].
			$$
		\end{enumerate}
		Here $\mathcal{A} \subset \R_{>0}$ has an accumulation point at $0$.
	\end{definition}
	
	\begin{remark}
		In comparison with the literature (e.g., \cite{LiWe20}), we make here two additional but rather weak assumptions, namely, the nonnegativity of the filter $\flt_{\alpha}(\cdot) \ge 0$, and the non-expansiveness $C_{\flt''} = 1$. In particular, \cref{d:of} still includes common examples of regularization methods, such as spectral cutoff, Tikhonov, iterated Tikhonov, Landweber, and Showalter, etc.
	\end{remark}
		
	\subsubsection{Spectral source condition}
	
	For the analysis of filter-based regularizations, further conditions on $\sigdag$ are necessary. Despite the discussion in the Introduction, we will work with a more general setting here (i.e., with general source conditions and general decay of the singular values of the forward operator). 
	
	\begin{assumption}[Spectral source condition]\label{a:scc}
		\begin{enumerate}
			\item
			Assume the source condition in \eref{e:scd} holds with {index function} $\varphi:\R_{\ge 0} \to \R_{\ge 0}$  being continuous, strictly increasing and $\varphi(0) = 0$.
			\item
			We assume 
			$$
			\phi:\R_{\ge 0} \to \R_{\ge 0}\qquad x \mapsto x \varphi^{-1}(\sqrt{x})\qquad \mbox{is convex,}
			$$
			where $\varphi^{-1}(\cdot)$ is the inverse function of $\varphi(\cdot)$.
		\end{enumerate}
	\end{assumption}
	
	\cref{a:scc}(ii) is made only for technical convenience, and imposes no restriction, as one can always work with a convex surrogate of $\phi$ in case that this assumption is violated. In particular, the index function in \eref{e:indfun} satisfies \cref{a:scc}. 
	
	We further require some compatibility between ordered filters and smoothness characterizations. 
	
	\begin{assumption}[Qualification]\label{a:qlf}
		The index function $\varphi$ in \eref{e:scd} is a \emph{qualification} of ordered filter $\flt_{\alpha}$ in the sense that
		$$
		\sup_{x \in [0,\; \norm{\fop^*\fop}]} \varphi(x) \bigl({1- x \flt_{\alpha}(x)}\bigr) \le C_{\varphi} \, \varphi(\alpha)\qquad \mbox{for}\quad \alpha \in \mathcal{A}.
		$$
	\end{assumption}
		
	As a consequence of \cref{d:of}(i), and Assumptions~\ref{a:scc}(i) and~\ref{a:qlf}, it holds that $\lim_{\alpha \to 0} x\,\flt_{\alpha}(x) = 1$ for all $x \in [0,\; \norm{\fop^*\fop}]$. It further follows that $\flt_\alpha(\fop^*\fop) \to (\fop^*\fop)^{-1}$ as $\alpha \to 0$. Thus, the filter-based regularization can be seen as a stable approximation of the least squares estimate. 
	
	In this paper, the notation $\lim_{\alpha \to 0}$ or $\alpha \to 0$ means that $\alpha$ goes to zero while remains in $\mathcal{A}$. This is possible, since $0$ is an accumulation point of~$\mathcal{A}$. 
		
	\subsection{SOLIT}
		
	Let us now introduce the SOLIT rule in detail.
	
	\subsubsection{Candidates for the regularization parameter $\alpha$}\label{ss:alpha}
	
	As many other parameter choice methods in the literature, our method relies heavily on a suitably chosen set of candidate parameters $\alpha_ 0 > \alpha_1 > \cdots > \alpha_{m_{\max}} > 0$. Generally, the value of regularization parameter can be interpreted as some sort of \enquote{degree of model complexity}: the smaller $\alpha_j$ (and hence the larger $j$), the smaller the bias and the larger the variance of the investigated estimator $\hat f_{\alpha_j}$. 
	
	A standard way is to consider geometrically equispaced candidates, i.e.\ 
	\begin{equation*}
		\alpha_m \equiv [\sigma^2 {\theta}^{m_{\max} - m}]_{\mathcal{A}}, \qquad m = 0, 1, \ldots, m_{\max},
	\end{equation*}
	where $[x]_{\mathcal{A}} := \argmin_{z \in \mathcal{A}}\abs{z - x}$ and $\theta > 1$ is a tuning parameter. The maximal index  $m_{\max}$ is determined to ensure that the smallest candidate $\alpha_{\max}$ is of the same order as $\sigma^2$. The \enquote{rounding} function $[x]_{\mathcal{A}}$ is necessary to ensure that only regularization parameters out of the possible set $\mathcal A$ are considered as candidates. For instance, in case of spectral cutoff, we have that $\alpha_m \in \{\lambda_1,\lambda_2,\ldots\}$ as $[x]_{\mathcal{A}} = \lambda_{k_x}$ with $k_x = \argmin_{k\in \N}\abs{x - \lambda_k}$.
	
	For the SOLIT, we proceed differently and assume that we are able to discretize $\mathcal A$ in a way that not the candidate parameters $\alpha_j$ decrease geometrically, but the variance of $\hat f_{\alpha_j}$ does. Let us therefore introduce the function
	\begin{equation}\label{eq:V_def}
		V_{T,\flt}\left(\alpha\right) := \frac{\trace\bigl(\mathrm{Var}(\hat\sig_{\alpha})\bigr)}{\sigma^2} = \sum_{k = 1}^\infty\lambda_k\bigl( \flt_{\alpha}(\lambda_k)\bigr)^2, \qquad \alpha \in \mathcal A,
	\end{equation}
	which is independent of the noise level $\sigma$ and depends only on the operator $T$ (via its singular values) and the ordered filter $\flt_\alpha$.
	
	We briefly collect some properties of $V_{T,\flt}$:
	
	\begin{lemma}\label{lem:v_T_flt}
		Let $\flt_{\alpha}$ be an ordered filter such that some index function $\varphi$ is a qualification. Then the function $V_{T,\flt} : \mathcal A \to \R$ obeys the following properties:
	\begin{enumerate}
		\item $V_{T,\flt}$ is decreasing.
		\item $	\lim_{\alpha \to 0} V_{T,\flt}\left(\alpha\right) = \infty$.
	\end{enumerate}
	\end{lemma}

\begin{proof}
\begin{enumerate}
	\item This follows immediately from \cref{d:of}(ii).
	\item By \cref{a:qlf} it holds $\lim_{\alpha\to 0} \flt_\alpha\left(\lambda\right) = \lambda^{-1}$. By Fatou's lemma we obtain
	\[
			\lim_{\alpha \to 0} V_{T,\flt}\left(\alpha\right)\geq \sum_{k = 1}^\infty\lambda_k\bigl(\lim_{\alpha \to 0}  \flt_{\alpha}(\lambda_k)\bigr)^2 = \sum_{k = 1}^\infty\frac{1}{\lambda_k} = \infty.
	\]
\end{enumerate}	
\end{proof}

	We now suppose the following:
	\begin{assumption}[Candidate selection]\label{ass:alpha}
		Let $\theta_1 > 1$ be fixed. Suppose there exist a sequence of parameters 
		\[
		\alpha_ 0 > \alpha_1 > \cdots > \alpha_n > \cdots
		\]
		in $\mathcal A$ and a universal constant $\theta_2 \geq \theta_1$ such that 
		\begin{equation}\label{eq:V}
			\theta_1 \leq \frac{V_{T,\flt}\left(\alpha_m\right)}{V_{T,\flt}\left(\alpha_{m-1}\right)} \leq \theta_2 \qquad \mbox{for all}\quad m \in \mathbb N. 
		\end{equation}
	\end{assumption}
	
	Assumption \ref{ass:alpha} is satisfied in nearly all practically relevant examples, especially under the following mild conditions on the filter $\flt$:
	\begin{lemma}\label{lem:alpha}
		Let $\mathcal A = \left(0,\infty\right)$ and $\alpha \mapsto \flt_\alpha\left(\lambda\right)$ be continuous for all $\lambda \in \left[0,\norm{T^*T}\right]$. Then for any fixed $\theta_1 > 1$, Assumption \ref{ass:alpha} is satisfied with $\theta_2 = \theta_1$.
	\end{lemma}
	\begin{proof}
		As discussed above, the function $V_{T,\flt}$ is decreasing. Furthermore, the series in the definition \eref{eq:V_def} converges uniformly over all compact sets of $\left(0,\norm{T^*T}\right]$ according to the first upper bound in \cref{d:of}(i). As $\alpha \mapsto \flt_\alpha\left(\lambda\right)$ is assumed to be continuous for fixed $\lambda$, this shows that $V_{T,\flt}$ as a uniformly convergent sum over continuous functions is also continuous itself on $\left(0,\norm{T^*T}\right]$.
		
		Now we set $\alpha_0 > 0$ arbitrarily and generate the sequence $\left(\alpha_n\right)_{n \in \mathbb N}$ recursively by
		\[
		\alpha_n = \max\set{ \alpha < \alpha_{n-1}}{ V_{T,\flt} \left(\alpha\right)= \theta_1 V_{T,\flt}\left(\alpha_{n-1}\right)}, \qquad n \in \mathbb N.
		\]
		As $V_{T,\flt}$ is continuous and obeys \cref{lem:v_T_flt}(ii), this maximum is attained by the intermediate value theorem. This choice clearly obeys \eref{eq:V} with $\theta_2 = \theta_1$.
	\end{proof}
	The above lemma, in particular, implies that Assumption \ref{ass:alpha} is satisfied if we consider filters such as Tikhonov, iterated Tikhonov or Showalter. 
	
	Given a sequence $\left(\alpha_n\right)_{n \in \mathbb N}$ satisfying \cref{ass:alpha}, we now determine our candidates as $\alpha_ 0 > \alpha_1 > \cdots > \alpha_{m_{\max}}$ with 
	\[
	m_{\max} = \min \set{m \in \mathbb N}{ \sigma^2 V_{T,\flt}\left(\alpha_m\right) \geq 1}.
	\]
	By \cref{lem:v_T_flt}(ii) it is ensured that $m_{\max}$ exists. Note that $m_{\max}\to \infty$ as $\sigma\to 0$. Let us denote 
	\begin{subequations}
	\begin{equation}\label{e:gsp:a}
		\mathcal{A}_{\theta_1,\theta_2} = \set{\alpha_m}{m = 0, 1, \ldots, m_{\max}}.
	\end{equation}
	Then it holds that
	\begin{eqnarray}
		&\infty > \theta_2 \ge \frac{\vap_{m}}{\vap_{m-1}} \ge \theta_1 > 1  \qquad \mbox{for all}\quad m = 1, \ldots, m_{\max}, \label{e:gsp:b}\\
		&\mbox{with}\quad
		\vap_{m} := \sigma^2 V_{T, \flt}(\alpha_m)
		\label{e:gsp:c}
	\end{eqnarray}
	\end{subequations}
	and $\vap_{m_{\max}} \approx 1$ in our asymptotic considerations as $\sigma \to 0$. Ideally, $\alpha_0$ should be chosen such that $\vap_0\approx \sigma^2$ to ensure a good performance in practice, but this does not influence our asymptotic results at all, as $\alpha_0$ is a fixed constant.  
	
	Besides $\vap_{m}$ (shorthand notation for $V_{T,\flt}(\alpha_m)$ up to $\sigma^2$), we introduce another size of variance as 
	$$
	\vop_m :=  \sigma^2\max_{k \in \N} \lambda_k \flt_{\alpha_m}(\lambda_k)^2.
	$$
	{\rev
 This measures the \enquote{amplification} effect, if we view the covariance $$\mathrm{Var}\bigl(\hat\sig_{\alpha_{m}}\bigr) = \sigma^2\flt_{\alpha_m}(\fop^*\fop)\fop^*\fop\flt_{\alpha_m}(\fop^*\fop) = \sigma^2 T^*T \flt_{\alpha_m}(T^*T)^2$$ as a linear operator on $\X$. More precisely, it holds that $\vop_m = \bigl\Vert{\mathrm{Var}\bigl({\hat\sig_{\alpha_{m}}}\bigr)}\bigr\Vert$, where $\norm{\cdot}$ denotes the operator norm.}

	\subsubsection{The SOLIT rule}
	
	Now let $\beta > 0$ be a fixed parameter, e.g., $\beta = 1$. In the end, $\beta$ can be seen as a tuning parameter of SOLIT, even though the results proven later hold independent of its specific value.
	
	For $m_1,m_2 \in \{0, 1, \ldots, m_{\max}\}$, we define the \emph{deterministic} part of the comparison between $m_1$ and $m_2$ as
	$$
	\bip_{m_1,m_2} := \norm{\E{\hat{\sig}_{\alpha_{m_1}} - \hat{\sig}_{\alpha_{m_2}}}} = \norm{{\sig}_{\alpha_{m_1}}-{\sig}_{\alpha_{m_2}}}, \quad \mbox{with } \sig_\alpha = \flt_\alpha(\fop^*\fop) \fop^*\fop \sigdag,
	$$
	the \emph{random} part as  
	$$
	\nop_{m_1,m_2} := \norm{\err_{\alpha_{m_1}} - \err_{\alpha_{m_2}}},\quad\mbox{with }\err_{\alpha} := \flt_\alpha \left(\fop^*\fop\right) \fop^*\err,
	$$
	and the \emph{variance} (arising from the random part) as
	$$
	\vap_{m_1, m_2} := \trace\Bigl(\mathrm{Var}\bigl({\hat\sig_{\alpha_{m_1}} - \hat\sig_{\alpha_{m_2}}}\bigr)\Bigr)  = \sigma^2\E{\nop_{m_1,m_2}^2} = \sigma^2 \sum_{k = 1}^\infty\lambda_k \bigl(\flt_{\alpha_{m_1}}(\lambda_k)-\flt_{\alpha_{m_2}}(\lambda_k)\bigr)^2.
	$$
	These terms arise naturally from the identity 
    $$
    \hat{\sig}_{\alpha_{m_1}} - \hat{\sig}_{\alpha_{m_2}}\; = \; \left({\sig}_{\alpha_{m_1}} - {\sig}_{\alpha_{m_2}}\right) \,+\,\sigma\left( \err_{\alpha_{m_1}} - \err_{\alpha_{m_2}}\right).
    $$
	Note that $\bip_{m_1,m_2} = \bip_{m_2,m_1}$, $\nop_{m_1,m_2} = \nop_{m_2,m_1}$ and $\vap_{m_1,m_2} = \vap_{m_2,m_1}$, but we often use such notation with the convention that the first index is smaller than the second index. Given $m_1 \le m_2$, it holds that  $\alpha_{m_1} \ge \alpha_{m_2}$ and $0 \le \flt_{\alpha_{m_1}} (\lambda_k)\le \flt_{\alpha_{m_2}}(\lambda_k)$ by the filter properties in \cref{d:of}. Thus, 
	$$
	\vap_{m_1, m_2} = \sigma^2 \sum_{k = 1}^\infty\lambda_k \bigl(\flt_{\alpha_{m_2}}(\lambda_k)-\flt_{\alpha_{m_1}}(\lambda_k)\bigr)^2 \le \sigma^2 \sum_{k = 1}^\infty\lambda_k \flt_{\alpha_{m_2}}(\lambda_k)^2 = \vap_{m_2} < \infty.
	$$
	
	Similar to the spirit of standard Lepski\u{\i}-type  principles, we introduce {\rev an} \emph{oracle} choice $\alpha_* := \alpha_{m_*}$, {\rev which} is defined as 
	\begin{equation}\label{e:orcr}
		m_* := \min\set{m\in \N_0}{\max_{m_2 > m_1 \ge m}\bigl(\bip_{m_1,m_2}^2 - \beta^2\,\vap_{m_1,m_2}\bigr)\le 0},
	\end{equation}
	{\rev see also \cite{SpWi19}; If the set in \eref{e:orcr} is empty, we define $m_* = m_{\max}$.} Note that $m_*$ is deterministic, and independent of the noise $\err$. However, this oracle rule is practically useless, since $\bip_{m_1,m_2}$ is not accessible. A natural modification is to replace $\bip_{m_1,m_2}$ by its empirical counterpart 
	$$
	\hat{\bip}_{m_1, m_2}:= \snorm{{\hat{\sig}_{\alpha_{m_1}} - \hat{\sig}_{\alpha_{m_2}}}}.
	$$
	This leads to the following rule of parameter choice. 
	\begin{definition}[SOLIT]\label{d:solit}
		Given a decreasingly ordered set of candidate parameters $\{\alpha_0 > \alpha_1 > \cdots > \alpha_{m_{\max}}\}$, the  \emph{Sharp Optimal Lepski\u{\i}-Inspired Tuning} (SOLIT) rule $\hat\alpha := \alpha_{\hat m}$ is defined as 
		\begin{subequations}
		\begin{equation}\label{e:limit:a}
			\hat m := \min\set{m_1\in \N_0}{\max_{m_2 > m_1} \bigl(\hat{\bip}_{m_1, m_2} - \thd_{m_1,m_2}\bigr) \le 0},
		\end{equation}
		where $\thd_{m_1, m_2}\in \R_{\ge 0}$ is some proper threshold given by
		\begin{equation}\label{e:limit:b}
			\thd_{m_1, m_2} = \sigma z_{m_1,m_2}(x_{m_1}) + \beta \sqrt{\vap_{m_1, m_2}},
		\end{equation}
		with
		\begin{equation}\label{e:limit:c}
			z_{m_1,m_2}(x)\quad \mbox{such that} \quad\Prob{\nop_{m_1,m_2} > z_{m_1,m_2}(x)} = e^{-x},
		\end{equation}
		and
		\begin{equation}\label{e:limit:d}
			x_m = 2(1+\gamma)\log \frac{\vap_{m+1}}{\vap_0}\qquad \mbox{for}\quad \gamma > 0.
		\end{equation}
		\end{subequations}
	\end{definition}
	
	\subsection{Adaptive minimax convergence rates via SOLIT}\label{sec:solit_rates}
	
	With the minimax paradigm from \cref{s:moa}, we are now in position to derive results for filter based regularization \eref{e:ualhat} with $\alpha$ chosen according to SOLIT. 
		
	\subsubsection{An oracle inequality}
	Using the results from \cite{SpWi19}, we can now state the following oracle inequality:
	\begin{theorem}[Oracle inequality~\cite{SpWi19}]\label{th:sw19}
		Assume the model in~\emph{\eref{e:model}}, and that Assumptions~\ref{a:scc} and~\ref{a:ss} hold. Let $\hat\sig_\alpha = \flt_\alpha(\fop^*\fop) \fop^* \data
		$ with $\flt_\alpha(\cdot)$ an ordered filter and let the oracle choice $\alpha_{m_*}$ in \emph{\eref{e:orcr}} and the SOLIT rule $\alpha_{\hat m} \in \mathcal{A}_{\theta_1, \theta_2}\equiv \{\alpha_0, \ldots, \alpha_{m_{\max}}\}$ in \emph{\eref{e:gsp:a}}--\emph{\eref{e:gsp:c}} with $\hat m \equiv \hat m(\beta, \gamma)$ in \cref{d:solit}. Then, it holds that
		$$
		\E{\norm{\hat \sig_{\alpha_{\hat m}} - \sigdag}^2} \le C_{1,m_*} \risk_{m_*} + \left(\sqrt{\risk_{m_*}} + C_{2,m_*} \right)^{2}, 
		$$
		where $\risk_{m_*} := \sE{\snorm{\hat \sig_{\alpha_{m_*}} - \sigdag}^2}$, and
		\begin{eqnarray*}
			&C_{1,m_*} &:= \frac{2\sqrt{3}}{\theta_1^\gamma - 1}\left(\frac{\vap_0}{\vap_{m_*}}\right)^{1+ \gamma}\; \le\; \frac{2\sqrt{3}}{\theta_1^\gamma - 1},\\
			\mbox{and}\quad 
			&C_{2,m_*} &:= \beta \sqrt{\vap_{m_*}} + \sqrt{2\vop_{m_*}\left(2(1+\gamma)\log\frac{\vap_{m_*}}{\vap_{0}} + \log(1+ m_{\max}) \right)}.
		\end{eqnarray*}
	\end{theorem}
	
	\begin{proof}
		This can be shown by casting the abstract and general result developed in \cite{SpWi19} into our particular setup. To this end, we only need to check the following two conditions:
		\begin{enumerate}[label={(\roman*)}]
			\item \label{as:ob}
			Decreasingly ordered bias, i.e.
			$$
			\norm{\sig_{\alpha_m} - \sigdag} \le \norm{\sig_{\alpha_{m_*}} - \sigdag}\qquad \mbox{for all}\quad m > m_*.
			$$
			\item \label{as:gv}
			Exponentially growing variance, i.e.
			$$
			\sum_{m > m_*} \left(\frac{\vap_{m_*}}{\vap_{m}}\right)^\gamma \le C < \infty,
			$$
			with $C$ a constant independent of $m_*$ and $\gamma$.
		\end{enumerate}
		Note that the properties \eref{flt:upb} and \eref{flt:ord} of ordered filter $\flt_{\alpha}(\cdot)$ in \cref{d:of}, in particular, the assumptions $C_q'' = 1$ and $\flt_{\alpha}(\cdot) \ge 0$, imply that 
		$$
		0 \le 1 - \lambda_k q_{\tilde{\alpha}}(\lambda_k) \le 1 - \lambda_k q_{\alpha}(\lambda_k)\qquad \mbox{for all}\quad k \in \N \quad\mbox{and}\quad \tilde{\alpha} < \alpha.
		$$
		Thus, the condition \ref{as:ob} above is satisfied. 
		
		From \eref{e:gsp:b} it follows that 
		$$
		\frac{\vap_{m_*}}{\vap_{m_* + k}} = \frac{\vap_{m_*}}{\vap_{m_* + 1}} \cdots \frac{\vap_{m_* + k-1}}{\vap_{m_* + k}} \le \left(\frac{1}{\theta_1}\right)^k, \qquad k \in \N.
		$$
		Further, noting $\theta_1 > 1$, we have  
		$$
		\sum_{m > m_*} \left(\frac{\vap_{m_*}}{\vap_{m}}\right)^\gamma \le \sum_{k = 1}^\infty \left(\frac{1}{\theta_1}\right)^{k\gamma} = \frac{1}{\theta_1^\gamma - 1}, 
		$$
		i.e.\  the condition \ref{as:gv} above is also satisfied for every $\gamma > 0$. 
		
		Then, the assertion of this theorem follows immediately from Theorem B.1 and Proposition B.2 in the {\rev supplement} of \cite{SpWi19}.
	\end{proof}
	
	\begin{remark}\label{r:c2ms}
		In {\rev many} cases, we have $\vap_0 / \vap_{m_*} \to 0$ as $\sigma\to 0$, and then $C_{1,m_*} \to 0$. Thus, the price that one needs to pay for the SOLIT rule, in comparison with the oracle choice, is asymptotically determined by $C_{2,m_*}$. The first term in $C_{2,m_*}$ can be bounded using the fact that $\vap_{m_*} \le \risk_{m_*}$. Note that $\vop_m$ are increasing with respect to $m$, and we can thus bound  the second term in $C_{2,m_*}$ by bounding $m_*$ from above. The term $\log(1+m_{\max})$ in $C_{2,m_*}$ is of order $\log \log (1/\sigma^2)$ by \eref{e:gsp:b}, thus being asymptotically negligible in usual situations. 
	\end{remark}
	
	\subsubsection{Oracle risks}
	Next we recall some results on the convergence rates of a priori choice rules for parameters in the literature. To state those, we require a rather general assumption on the decay behaviour of the singular values. Functionals of $\fop^*\fop$ can be determined by functionals on eigenvalues $\{\lambda_k\}_{k \in \N}$ of $\fop^*\fop$, for instance, $\trace(\fop^*\fop) = \sum_{k=1}^\infty \lambda_k$. Calculations that involve summations over $\{\lambda_k\}_{k \in \N}$ can be formulated as Lebesgue--Stieltjes integrals with respect to the counting measure
	\begin{equation}\label{e:cm}
		\Sigma(x) := \Sigma\bigl([x, \infty)\bigr):= \#\set{k\in \N}{\lambda_k \ge x}.
	\end{equation}
	We employ the techniques developed in \cite{BHMR07} to tackle the difficulty caused by the discontinuity of $\Sigma(\cdot)$ and assume:
 {\rev 
 \begin{assumption}[Smooth surrogate]\label{a:ss}
		\begin{enumerate}
			\item \label{a:ss1}
            There exists a continuous surrogate function $S: (0, \lambda_1] \to [0, \infty)$ of $\Sigma(\cdot)$ in \eref{e:cm} satisfying
			\begin{eqnarray*}
				\liminf_{\alpha \to 0}\frac{S(\alpha) }{ \Sigma(\alpha)} > 0\quad\mbox{and}\quad  \lim_{\alpha \to 0} \alpha S(\alpha) = 0.
			\end{eqnarray*}
			\item\label{a:ss2}
			There exists a constant $C_S > 0$ such that 
			$$
			\frac{1}{\alpha}\int_0^\alpha \Sigma(t) \diff t \le C_{S}S(\alpha) \qquad \mbox{ for all } \quad \alpha \in \mathcal{A}. 
			$$
		\end{enumerate} 
	\end{assumption}}
		 
	\begin{lemma}[Bias--variance decomposition~\cite{BHMR07}]\label{l:bvd}
		Assume the model in~\eref{e:model} and suppose that Assumptions~\ref{a:scc}, \ref{a:qlf} and~\ref{a:ss} hold. Let $\hat\sig_\alpha = \flt_\alpha(\fop^*\fop) \fop^* \data
		$ with $\flt_\alpha(\cdot)$ an ordered filter. Then there exist constants $C_{\rm b}$ and $C_{\rm v}$ such that
		$$
		\sup_{\sigdag \in  \mathcal{W}_{\varphi} (\rho)} \E{\norm{\hat \sig_\alpha - \sigdag}^2} \le C_{\rm b} \varphi(\alpha)^2 + C_{\rm v} \sigma^2 \frac{S(\alpha)}{\alpha} \qquad \mbox{for}\quad \alpha \in \mathcal{A}.
		$$
		More precisely, $C_{\rm b} := C_{\varphi}^2\rho^2$ and $C_{\rm v}:= C_S\max\{C_q',C_q''\}^2$,
  with $C_q'$ and $C_q''\equiv 1$ as in \cref{d:of}, and $C_{\varphi}$ in \cref{a:qlf}. 
	\end{lemma}
	
	\begin{proof}
		It follows essentially from Theorem~3 in \cite{BHMR07}. {\rev As we slightly simplify (or modify) the requirement on the smooth surrogate of \cite{BHMR07}, we provide below the upper bound on the variance part for completeness. Let $C_q := \max\{C_q', C_q''\}$. Then, by \cref{d:of},
  \begin{eqnarray*}
      \E{\norm{\hat \sig_\alpha - \E{\hat\sig_{\alpha}}}^2} & = \sigma^2 \sum_{k =1}^{\infty} \lambda_k \flt_{\alpha}(\lambda_k)^2 = -\sigma^2\int_{0}^\infty x\flt_{\alpha}(x)^2\mathrm{d}\Sigma(x)\\
      & \le C_q^2\sigma^2\left(-\frac{1}{\alpha^2}\int_{0}^{\alpha}x\mathrm{d}\Sigma(x) - \int_{\alpha}^{\infty}\frac{1}{x}\mathrm{d}\Sigma(x) \right).
  \end{eqnarray*}
  Note that $\lim_{\alpha \to 0} \alpha \Sigma(\alpha) = \lim_{\alpha \to 0}\bigl(\alpha S(\alpha)\bigr)\bigl( \Sigma(\alpha)/ S(\alpha)\bigr) = 0 $, by \cref{a:ss}\eref{a:ss1}. Further, we apply partial integration and obtain
\begin{eqnarray*}
\E{\norm{\hat \sig_\alpha - \E{\hat\sig_{\alpha}}}^2} & \le C_q^2\sigma^2\left(-\frac{1}{\alpha^2}\int_{0}^{\alpha}x\mathrm{d}\Sigma(x) - \frac{1}{\alpha}\int_{\alpha}^{\infty}\mathrm{d}\Sigma(x) \right)\\
& = C_q^2\sigma^2\left(-\frac{\Sigma(\alpha)}{\alpha} + \frac{1}{\alpha^2}\int_{0}^{\alpha}\Sigma(x)\mathrm{d}x + \frac{\Sigma(\alpha)}{\alpha}\right)\\
& \le C_SC_q^2\sigma^2\frac{S(\alpha)}{\alpha},
\end{eqnarray*}
where the last inequality is due to \cref{a:ss}\eref{a:ss2}.}
	\end{proof}
	
	\begin{remark}\label{r:so}
		In fact, the term $C_{\rm b} \varphi(\alpha)^2$ in \cref{l:bvd} is an upper bound for the squared bias part $\snorm{\sig_\alpha-\sigdag}^2$ and the term $C_{\rm v} \sigma^2 {S(\alpha)}/{\alpha}$ is for the variance part $\sE{\snorm{\hat\sig_\alpha-\sig_\alpha}^2}$. {\rev In order to derive an as sharp upper bound for the risk as possible, one is often searching for the smallest possible smooth surrogate $S(\cdot)$ of $\Sigma(\cdot)$. For such $S(\cdot)$, we usually have $\limsup_{\alpha\to 0}S(\alpha)/\Sigma(\alpha) < \infty$, even though this is not required in \cref{a:ss} (cf.\ \cref{s:rates}). In this situation,} the bounds in \cref{l:bvd} are often sharp up to multiplying constants, since they lead to minimax optimality in order in many situations, as detailed in \cite{BHMR07}.
	\end{remark}
	As a consequence, we can find {\rev a reasonable} $\alpha$ by minimizing the upper bound of the risk in  \cref{l:bvd}. Recall that we consider candidate parameters in $\mathcal{A}_{\theta_1, \theta_2} = \{\alpha_0, \ldots, \alpha_{m_{\max}}\}$, a discrete subset of $\mathcal{A}$. Thus, we define
	\begin{equation}\label{e:orcm}
		m_{\diamond} := \argmin_{m\in \N_0}\left(C_{\rm b} \varphi(\alpha_m)^2 + C_{\rm v} \sigma^2 \frac{S(\alpha_m)}{\alpha_m}\right).
	\end{equation}
	Based on it, the ordered filters are able to achieve (up to possible logarithmic factors) the minimax optimal rates over various smoothness classes for mildly and severely ill-posed problems (cf.\ \cite{BHMR07}). Recall, however, that the performance of SOLIT is evaluated with respect to a different oracle rule defined in~\eref{e:orcr}. Thus, an important step in deriving the explicit rates for ordered filters with the SOLIT rule in particular situations is to investigate the relations between these two oracle rules. 
	
	\begin{theorem}[Comparison of oracles]\label{t:orc}
		Assume the model in~\eref{e:model} and suppose that Assumptions~\ref{a:scc}, \ref{a:qlf} and~\ref{a:ss} hold.  Let $\hat\sig_\alpha = \flt_\alpha(\fop^*\fop) \fop^* \data
		$ with $\flt_\alpha(\cdot)$ an ordered filter. Let also $m_*$ be defined in \eref{e:orcr}, and $m_{\diamond}$ in \eref{e:orcm}. Then, $m_* \le \log_{\theta_1}\bigl({C_1 C_2 S(\alpha_{m_{\diamond}})}/{\alpha_{m_\diamond}}\bigr)$ and, for $\sigdag \in  \mathcal{W}_{\varphi} (\rho)$,
		$$
		\risk_{m_*} \le C_1(2\beta^2 + 4)\risk_{m_{\diamond}} \le C_1 (2\beta^2 + 4) \min_{\alpha \in \mathcal{A}_{\theta_1, \theta_2}} \left(C_{\rm b} \varphi(\alpha)^2 + C_{\rm v} \sigma^2 \frac{S(\alpha)}{\alpha}\right),
		$$
		where $\risk_{m} = \sE{\snorm{\hat \sig_{\alpha_{m}} - \sigdag}^2}$ for $m \in \N_0$, $C_1=1+\max\{(\sqrt{\theta_1}-1)^2\beta^2, \, \theta_2 (\sqrt{\theta_1}-1)^{-2}\beta^{-2}\}$ and $C_2$ is a constant depending only on $\rho, \varphi,S, C_q'$.
	\end{theorem}

	\begin{proof}
		Recall that $\mathcal{A}_{\theta_1, \theta_2} = \{\alpha_0, \ldots, \alpha_{m_{\max}}\}$ in \eref{e:gsp:a}--\eref{e:gsp:c}. We define
		$$
		m_{{**}} := \min\set{m \in \N_0}{\bip_{m}^2 \le (\sqrt{\theta_1} - 1)^2\beta^2 \vap_{m}}
		$$
		where $\bip_m := \norm{\sig_{\alpha_m} - \sigdag}$. Note that $\bip_m$ is decreasing (see the proof of \cref{th:sw19}), while $\vap_m$ is increasing, with respect to $m$. Thus, $m_{**}$ is well-defined. 
		
		We split the proof into four steps.
		
		Step 1. We will show $m_* \le m_{{**}}$. To this end, we consider $m_2 > m_1 > m_{{**}}$. The ordered property and $C_q'' = 1$ in \cref{d:of} imply that 
		$$
		\bip_{m_1,m_2}^2 = \norm{\sig_{\alpha_{m_2}} - f_{\alpha_{m_1}}}^2 \le \norm{\sigdag - \sig_{\alpha_{m_1}}}^2 = \bip_{m_1}^2 \le (\sqrt{\theta_1} - 1)^2\beta^2 \vap_{m_1},
		$$
		where the last inequality is due to $m_1 > m_{{**}}$ and the definition of $m_{{**}}$. 
		
		By the Cauchy--Schwarz inequality, we have 
		\begin{eqnarray*}
			\vap_{m_1,m_2} & = \sigma^2 \sum_{k = 1}^\infty\lambda_k \bigl(\flt_{\alpha_{m_1}}(\lambda_k)-\flt_{\alpha_{m_2}}(\lambda_k)\bigr)^2\\
			& = \vap_{m_1} + \vap_{m_2}- 2\sigma^2 \sum_{k = 1}^\infty\lambda_k \flt_{\alpha_{m_1}}(\lambda_k)\flt_{\alpha_{m_2}}(\lambda_k)\\
			& \ge \vap_{m_1} + \vap_{m_2} - 2\sqrt{\vap_{m_1}\vap_{m_2}}\\
			& = \vap_{m_1}\bigl(\sqrt{\vap_{m_2}/\vap_{m_1}}-1\bigr)^2.
		\end{eqnarray*}
		Note further that $\vap_{m_2}/\vap_{m_1}\ge \theta_1 > 1$ from \eref{e:gsp:b}. Then, 
		$$
		\vap_{m_1,m_2} \ge  \vap_{m_1}\left(\sqrt{\frac{\vap_{m_2}}{\vap_{m_1}}}-1\right)^2 \ge(\sqrt{\theta_1}-1)^2 \vap_{m_1}.
		$$
		Combining the derived inequalities above, we obtain 
		$$
		\bip_{m_1,m_2}^2 \le (\sqrt{\theta_1} - 1)^2\beta^2 \vap_{m_1} \le \beta^2 \vap_{m_1,m_2}.
		$$
		Namely, $m_{{**}}$ lies in the set in \eref{e:orcr}, so $m_*\le m_{{**}}$.
		
		Step 2. We will show $\risk_{m_*} \lesssim \risk_{m_{**}}$. 
		Because of $m_* \le m_{**}$ and the definition of $m_*$ in \eref{e:orcr}, it holds that
		$$
		\E{\norm{\hat\sig_{\alpha_{m_*}}-\hat\sig_{\alpha_{m_{**}}}}^2} = \bip_{m_*, m_{**}}^2 + \vap_{m_*, m_{**}} \le (1+\beta^2)\vap_{m_*, m_{**}} \le (1+\beta^2)\vap_{m_{**}},
		$$
		where the last inequality is due to the ordered property and nonnegativity of $\flt_{\alpha}(\cdot)$, see \cref{d:of}. Further, by the triangle inequality, we obtain, for $\sigdag \in \mathcal{W}_{\varphi}(\rho)$,
		\begin{eqnarray*}
			\risk_{m_*} &= \E{\norm{\hat\sig_{\alpha_{m_*}}-\sigdag}^2} \\
			& \le  2\E{\norm{\hat\sig_{\alpha_{m_*}}-\hat\sig_{\alpha_{m_{**}}}}^2} + 2\E{\norm{\hat\sig_{\alpha_{m_{**}}}-\sigdag}^2} \\
			& \le  2(1+\beta^2)\vap_{m_{**}} + 2\risk_{m_{**}}\\
			& \le (2\beta^2 + 4)\risk_{m_{**}}.
		\end{eqnarray*}
		
		Step 3. We will show $\risk_{m_{**}} \lesssim \risk_{m_{\diamond}}$. We consider two cases:
		\begin{itemize}
			\item Case: $m_{\diamond} \ge m_{**}$. The definition of $m_{**}$ and the increasing of $\vap_m$ imply 
			\begin{eqnarray*}
				\risk_{m_{**}} = \vap_{m_{**}} + \bip_{m_{**}}^2 & \le \bigl(1 + (\sqrt{\theta_1} - 1)^2\beta^2\bigr)\vap_{m_{**}}\\
				& \le\bigl(1 +(\sqrt{\theta_1} - 1)^2\beta^2\bigr)\vap_{m_{\diamond}} \\
				& \le \bigl(1 +(\sqrt{\theta_1} - 1)^2\beta^2\bigr)\risk_{m_{\diamond}}. 
			\end{eqnarray*}
			\item Case: $m_{\diamond} < m_{**}$. The definition of $m_{**}$ and the decreasing of $\bip_m$ imply
			\begin{eqnarray*}
				\risk_{m_{**}} = \vap_{m_{**}} + \bip_{m_{**}}^2 & \le \theta_2 \vap_{m_{**} - 1} + \bip_{m_{**}-1}^2 \\
				& \le \left(\frac{\theta_2}{(\sqrt{\theta_1}-1)^2\beta^2} + 1 \right)\bip_{m_{**}-1}^2 \\
				& \le \left(\frac{\theta_2}{(\sqrt{\theta_1}-1)^2\beta^2} + 1 \right)\bip_{m_{\diamond}}^2 \\
				& \le \left(\frac{\theta_2}{(\sqrt{\theta_1}-1)^2\beta^2} + 1 \right)\risk_{m_{\diamond}},
			\end{eqnarray*}
			where the first inequality above is due to \eref{e:gsp:b}.
		\end{itemize}
		Thus, combining both cases, we have $\risk_{m_{**}} \le C_1 \risk_{m_{\diamond}}$ with $C_1 := 1+\max\{(\sqrt{\theta_1}-1)^2\beta^2, \, \theta_2 (\sqrt{\theta_1}-1)^{-2}\beta^{-2}\}$. This together with Step 2 and \cref{l:bvd} leads to the second part of the assertion. 
		
		Step 4. We will show $m_{*} \lesssim \log \bigl(S(\alpha_{m_\diamond})/\alpha_{\diamond}\bigr)$. By the definition of $m_{\diamond}$ in \eref{e:orcm}, and $\vap_0 \asymp \sigma^2$, we have
		$$
		C_{\rm b} \varphi(\alpha_{m_\diamond})^2 + C_{\rm v} \sigma^2 \frac{S(\alpha_{m_\diamond})}{\alpha_{m_\diamond}} \le C_2 \vap_0 \frac{S(\alpha_{m_\diamond})}{\alpha_{m_\diamond}},
		$$
		where $C_2$ is a constant depending only on $C_{\rm b}, C_{\rm v}, \varphi$ and $S$. The forms of $C_{\rm b}, C_{\rm v}$ in \cref{l:bvd} further imply that $C_1$ depends only on $\rho, \varphi,S$ and $C_q'$.
		
		Note that 
		$$
		\vap_{m_{**}} \le \risk_{m_{**}} \le C_1 \left(C_{\rm b} \varphi(\alpha_{m_\diamond})^2 + C_{\rm v} \sigma^2 \frac{S(\alpha_{m_\diamond})}{\alpha_{m_\diamond}}\right) \le C_1 C_2 \vap_0 \frac{S(\alpha_{m_\diamond})}{\alpha_{m_\diamond}},
		$$
		and, due to \eref{e:gsp:b}, that $\vap_{m_{**}} \ge \vap_0\theta_1^{m_{**}}$. Thus, 
		$$
		m_{*} \le m_{**} \le \log_{\theta_1}\frac{C_1C_2S(\alpha_{m_\diamond})}{\alpha_{m_\diamond}},
		$$
		which is the first part of the assertion. 
	\end{proof}
	
	\begin{remark}\label{r:vop}
    {\rev In usual cases, the upper bound of $\risk_{m_\diamond}$ in the form given in \cref{l:bvd} tends to $0$ as  $\sigma \to 0$. In particular, we have $\sigma^2S(\alpha_{m_\diamond}) /\alpha_{m_\diamond} \to 0$, and thus} $C_1C_2S(\alpha_{m_\diamond}) /\alpha_{m_\diamond} \le 1/\sigma^2$ for sufficiently small $\sigma$. This together with \cref{t:orc} implies that $m_* \le - \log_{\theta_1} (\sigma^2)$. Moreover, note that  $$
		\vop_m := \sigma^2\max_{k \in \N} \lambda_k \flt_{\alpha_m}(\lambda_k)^2 \le \sigma^2\max_{k \in \N } \flt_{\alpha_m}(\lambda_k) \le \frac{C_q'\sigma^2}{\alpha_m},
		$$
		where the inequalities are based on the filter property in \cref{d:of}(i). 
  Thus, the (possible) price of adaptation for SOLIT in \cref{th:sw19} is  {\rev an additional term of order $\sigma^2(-\log \sigma^2)/\alpha_{m_*}$.}
	\end{remark}
	
	\section{Adaptive minimax rates via SOLIT}\label{s:rates}
	
	Combining \cref{th:sw19,t:orc}, we can derive the minimax adaptation rates in various cases for the SOLIT rule $\hat \alpha$ in \cref{d:solit}. 

	\subsection{The mildly ill-posed case}

We further require a lower bound on the variance part that arises in the bias--variance decomposition (cf.\ \cref{l:bvd}). This is formulated as follows. 

	\begin{assumption}\label{a:lbv1}
		There exists a similar lower bound on $\vap_m$ as the upper bound in \cref{l:bvd}, namely, for some constant $c_{\rm v}$,
		$$
		\vap_m \;\ge \; c_{\rm v}\sigma^2\frac{S(\alpha_m)}{\alpha_m},\qquad\mbox{for}\quad m \in \N_0.
		$$
	\end{assumption}

Note that $\vap_m / \sigma^2 = V_{T,\flt} = \sum_{k = 1}^{\infty}\lambda_k\bigl( \flt_{\alpha}(\lambda_k)\bigr)^2$, and $\lambda_k \asymp k^{-a}$ by \eref{e:eigop}. One can show that \cref{a:lbv1} holds for common ordered filters, including spectral cut-off, (iterated) Tikhonov, Landweber and Showalter, with $S(x)\asymp x^{-1/a}$. For instance, in case of spectral cut-off, i.e., $\flt_\alpha(\lambda) = \lambda^{-1}\mathbf{1}\{\lambda \ge \alpha\}$, it holds that 
$$
\frac{\vap_m}{\sigma^2} \;= \;\sum_{k = 1}^{\infty}\frac{\mathbf{1}\{\lambda_k \ge \alpha\}}{\lambda_k}\; \gtrsim\; \int_{0}^{\infty} x^{a}\mathbf{1}\{x^{-a} \gtrsim \alpha\} \diff x \;\asymp\; \alpha^{-\frac{a + 1}{a}},
$$
which is $S(\alpha)/\alpha$ up to a multiplying constant. 

\begin{theorem}[Mildly ill-posed problems]\label{t:mip}
		Assume the model in~\eref{e:model}, and suppose that the eigenvalues of the forward operator decay at a polynomial rate as in \eref{e:eigop} with $a >1$. Let Assumptions~\ref{a:qlf} and \ref{a:lbv1} hold, where the index function $\varphi\equiv \varphi_{\nu,\tau}$ is given in \eref{e:indfun} with the parameters $\nu > 0$ and $\tau \in \R$, or $\nu = 0$ and $\tau > 0$. Let $\hat\sig_\alpha = \flt_\alpha(\fop^*\fop) \fop^* \data$ with an ordered filter $\flt_\alpha$, and consider SOLIT $\alpha_{\hat m} \in \mathcal{A}_{\theta_1, \theta_2}\equiv \{\alpha_0, \ldots, \alpha_{m_{\max}}\}$ in \emph{\eref{e:gsp:a}}--\emph{\eref{e:gsp:c}} with $\hat m \equiv \hat m(\beta, \gamma)$ as in  \cref{d:solit} as the parameter choice rule. Then, for some constant $C > 0$,
		$$
		\sup_{\sigdag \in  \mathcal{W}_{\varphi} (\rho)} \E{\norm{\hat \sig_{\alpha_{\hat m}} - \sigdag}^2} \le C (\sigma^2)^{\frac{2\nu}{2\nu+1/a+1}}(-\log \sigma^2)^{-\frac{2\tau(1+1/a)}{2\nu + 1+1/a}}.
		$$
\end{theorem}
	
	\begin{proof}
		It follows from the polynomial decay of $\{\lambda_k\}_{k\in\N}$ that there is a smooth surrogate function $S$ satisfying \cref{a:ss}, which takes the form of $S(x) \asymp x^{-1/a}$. Recall that the index function $\varphi$ in \eref{e:indfun} meets the requirement of \cref{a:scc}. Thus, applying \cref{l:bvd}, we obtain, with $m_\diamond$ in \eref{e:orcm},
		$$
		\alpha_{m_\diamond} \asymp (\sigma^2)^{\frac{1}{2\nu+1+1/a}}(-\log \sigma^2)^{\frac{2\tau}{2\nu + 1+ 1/a}},
		$$
		and
		$$
		\sup_{\sigdag \in \mathcal W_{\varphi}(\rho)} \E{\norm{\hat \sig_{\alpha_{m_{\diamond}}} - \sigdag}^2} \,\lesssim\, {\rev\sigma^2\frac{S(\alpha_{m_\diamond})}{\alpha_{m_\diamond}}\,\asymp\, }(\sigma^2)^{\frac{2\nu}{2\nu+1/a+1}}(-\log \sigma^2)^{-\frac{2\tau(1+1/a)}{2\nu + 1+1/a}}.
		$$
		Further, by \cref{t:orc} and \cref{a:lbv1}, it holds that 
		$$
		\sigma^2\frac{S(\alpha_{m_*})}{\alpha_{m_*}}\lesssim \vap_{m_*} \le \risk_{m_*} \lesssim \varphi(\alpha_{m_\diamond})^2 + \sigma^2\frac{S(\alpha_{m_\diamond})}{\alpha_{m_\diamond}} \asymp \sigma^2\frac{S(\alpha_{m_\diamond})}{\alpha_{m_\diamond}},
		$$
		which yields $\alpha_{m_\diamond} \lesssim \alpha_{m_*}$. Hence, applying again \cref{t:orc} (see also \cref{r:vop}), we have  $m_* \lesssim \log(1/\sigma^2)$ and $\vop_{m_*} \lesssim \sigma^2/\alpha_{m_*} \lesssim \sigma^2/\alpha_{m_\diamond}$. {\rev Note that $m_{\max} \asymp \log(1/\sigma^2)$ by \cref{a:lbv1} and the form of $S(\cdot)$, see also \cref{r:c2ms}. It follows that  $C_{2,m_*}$ in \cref{th:sw19} satisfies $C_{2,m_*}^2 \lesssim \sigma^2\log(1/\sigma^2)/\alpha_{m_\diamond} \ll \sigma^2{S(\alpha_{m_\diamond})}/{\alpha_{m_\diamond}}$.} Thus, the assertion of theorem follows now from \cref{th:sw19,t:orc}.
	\end{proof}
	
	\begin{remark}
		The convergence rate in \cref{t:mip} is minimax optimal for the function class $\mathcal{W}_{\varphi} (\rho)$, see {\rev e.g.\ Theorem~1 in} \cite{Pin80}. Note that the SOLIT rule does not require the exact smoothness of the truth, namely, the parameters $\nu$ and $\tau$ of the index function $\varphi$. Thus, the ordered filters with the SOLIT rule automatically adapt to the smoothness of the true. In particular, when $\nu > 0$ and $\tau \in \R$, it corresponds to the polynomial smoothness, and when $\nu = 0$ and $\tau > 0$, it corresponds to the logarithmic smoothness. In the latter case, the convergence rate in \cref{t:mip} is of order $(-\log \sigma^2)^{-2\tau}$, because of $\nu = 0$.
	\end{remark}
	
	\subsection{Severely ill-posed problems}\label{s:sip}
	
	In this section, we consider severely ill-posed problems, namely, that the eigenvalues $\lambda_k$ of $\fop^*\fop$ in \eref{e:eigdc} decay exponentially as in \eref{e:eigopb} for some $a \in \R$ and $b, \vartheta > 0$.
	
	\subsubsection{Technical challenges}\label{ss:tc}
	We start with the additional technical issues arising from severe ill-posedness. 
	
	The super fast decay of eigenvalues poses difficulty in the selection of candidate parameters (cf.~\cref{ss:alpha}): Finding parameters with geometrically changing variance may be impossible. For instance, in case of spectral cut-off, we have $\mathcal A = \left\{\lambda_1, \lambda_2, ...\right\}$ and consequently, $V_{T,\flt}$ can be seen as a function of $n \in \mathbb N$ such that
	\[
	V_{T,\flt} \left(n\right) = \sum_{k=1}^n \frac{1}{\lambda_k}.
	\]
	It can be readily seen that \cref{ass:alpha} can be fulfilled only if the eigenvalues $\lambda_k$ decay not faster than $k^{a}\exp\left(-b k^\vartheta\right)$ with parameters $b > 0$, $a < 0$ and $\vartheta \in \left[0,1\right]$. As soon as $\vartheta > 1$, we already have
	\[
	\lim_{n \to \infty} \frac{V_{T,\flt} \left(n\right)}{V_{T,\flt} \left(n-1\right)} = \infty,
	\]
	so that \eref{eq:V} cannot be valid for any sequence of regularization parameters in $\mathcal A$.

	Another difficulty is that the analysis technique of replacing summation over eigenvalues by integral with respect to surrogate function (cf.~\cref{a:ss}; originally developed in \cite{BHMR07}) are no longer sharp in rates, but the gap is only in log factors. More precisely, provided that $\lambda_k \asymp k^a \exp(-bk^\vartheta)$ with $b,\vartheta > 0$, the variance of spectral cut-off is of order 
	$$
	\sigma^2 \alpha (-\log \alpha)^{1/\vartheta-1}\qquad \mbox{for}\quad \alpha \in \mathcal{A} \equiv \{\lambda_1, \lambda_2, \ldots\},
	$$
	while the bound in \cref{l:bvd} gives that the variance part is of order $ \sigma^2 \alpha (-\log \alpha)^{1/\vartheta}.$ 
	In the severely ill-posed situation, a sharp bound on the variance part remains open for general (ordered) filters, and is an interesting topic for future research. 
	
	\subsubsection{Performance of SOLIT}
	As mentioned above, the upper bound on the variance part in \cref{l:bvd} is sub-optimal in case of serious ill-posedness. Thus, we need a weaker assumption than \cref{a:lbv1} as follows.
	\begin{assumption}\label{a:lbv2}
		We assume a lower bound on $\vap_m$ slightly smaller than the upper bound in \cref{l:bvd}, more precisely, for some constant $c_{\rm v}$,
		$$
		\vap_m \;\ge c_{\rm v} \; \sigma^2\frac{S(\alpha_m)}{(-\log \alpha_m)\alpha_m},\qquad \mbox{for}\quad m \in \N_0.
		$$
	\end{assumption}

Similar as \cref{a:lbv1} for mildly ill-posed problems, one can show that \cref{a:lbv2} is satisfied for common ordered filters, e.g., spectral cut-off, (iterated) Tikhonov, Landweber and Showalter, for severely ill-posed problems. As a showcase, we consider the spectral cut-off $\flt_\alpha(\lambda) = \lambda^{-1}\mathbf{1}\{\lambda \ge \alpha\}$. Then, by \eref{e:eigopb}, we obtain
$$
\frac{\vap_m}{\sigma^2} \,= \,\sum_{k = 1}^{\infty}\frac{\mathbf{1}\{\lambda_k \ge \alpha\}}{\lambda_k}\, \gtrsim\, \int_{0}^{\infty} x^{a} \exp(b x^{\vartheta})\mathbf{1}\{x^{-a} \exp(-b x^{\vartheta}) \gtrsim \alpha\} \diff x \,\asymp\, \frac{(-\log \alpha )^{\frac{1 - \vartheta}{\vartheta}}}{\alpha}. 
$$
Note that the exponential decay of eigenvalues in \eref{e:eigopb} implies that $S(x) \asymp (-\log x)^{1/\vartheta}$. Thus, \cref{a:lbv2} holds for the spectral cut-off. 
	
	\begin{theorem}[Severely ill-posed problems]\label{t:sip}
		Assume the model in~\eref{e:model}, and suppose that the eigenvalues of the forward operator decay at an exponential rate as in \eref{e:eigopb} with $a\in \R$, $b > 0$ and $0 < \vartheta \le 1/2$. Let Assumptions~\ref{a:qlf}, \ref{a:ss} and \ref{a:lbv2} hold, where the index function $\varphi\equiv \varphi_{\nu,\tau}$ is given in \eref{e:indfun} with the parameters $\nu > 0$ and $\tau \in \R$, or $\nu = 0$ and $\tau > 0$. Let $\hat\sig_\alpha = \flt_\alpha(\fop^*\fop) \fop^* \data$ with an ordered filter $\flt_\alpha$, and consider SOLIT $\alpha_{\hat m} \in \mathcal{A}_{\theta_1, \theta_2}\equiv \{\alpha_0, \ldots, \alpha_{m_{\max}}\}$ as in \emph{\eref{e:gsp:a}}--\emph{\eref{e:gsp:c}} with $\hat m \equiv \hat m(\beta, \gamma)$ as in \cref{d:solit} as the parameter choice rule. Then, for some constant $C > 0$,
		$$
		\sup_{\sigdag \in  \mathcal{W}_{\varphi_{\nu,\tau}} (\rho)} \E{\norm{\hat \sig_{\alpha_{\hat m}} - \sigdag}^2} \le C (\sigma^2)^{\frac{2\nu}{2\nu+1}}(-\log \sigma^2)^{\frac{2\nu/\vartheta-2\tau}{2\nu+1}}.
		$$
	\end{theorem}

	\begin{proof}
		It can be proven in exactly the same way as \cref{t:mip}, so, to avoid repetitions, we list only the differences in detailed computation: 
		\begin{itemize}
			\item The smooth surrogate function takes the form of  
			$$
			S(x) \;\asymp\; \left(-\log x + \frac{a}{\vartheta}\log (- \log x)\right)^{1/\vartheta}  \;\asymp \;(-\log x )^{1/\vartheta}.
			$$
			\item 
			The convergence rate corresponding to $m_\diamond$ in \eref{e:orcm} is
			$$
			\sup_{\sigdag \in \mathcal W_{\varphi}(\rho)} \E{\norm{\hat \sig_{\alpha_{m_{\diamond}}} - \sigdag}^2} \lesssim (\sigma^2)^{\frac{2\nu}{2\nu+1}}(-\log \sigma^2)^{\frac{2\nu/\vartheta-2\tau}{2\nu+1}},
			$$
			with $\alpha_{m_\diamond} \asymp (\sigma^2)^{\frac{1}{2\nu+1}}(-\log \sigma^2)^{\frac{2\tau+1/\vartheta}{2\nu+1}}$.
			\item {\rev By \cref{a:lbv2} and \cref{t:orc} we have
   $$
   \sigma^2\frac{S(\alpha_{m_*})}{(-\log\alpha_{m_*})\alpha_{m_*}}\lesssim \vap_{m_*} \le \risk_{m_*} \lesssim (\sigma^2)^{\frac{2\nu}{2\nu+1}}(-\log \sigma^2)^{\frac{2\nu/\vartheta-2\tau}{2\nu+1}},
   $$
   which together with the form of $S(\cdot)$ implies that (cf.\ \cref{r:vop})
   $$
   C_{2,m_*}^2 \lesssim m_* \vop_{m_*} \lesssim \frac{\sigma^2(-\log \sigma^2)}{\alpha_{m_*}} \lesssim (\sigma^2)^{\frac{2\nu}{2\nu+1}} (- \log\sigma^2)^{\frac{2\nu/\vartheta-2\tau}{2\nu+1} + 2- \frac{1}{\vartheta}}. 
   $$}
\end{itemize}
{\rev Thus, if $0 < \vartheta \le 1/2$, we have $\mathbb{E}[{\|{\hat \sig_{\alpha_{\hat m}} - \sigdag}\|^2}] \lesssim (\sigma^2)^{\frac{2\nu}{2\nu+1}}(-\log \sigma^2)^{\frac{2\nu/\vartheta-2\tau}{2\nu+1}}$ by \cref{th:sw19,t:orc}.}
	\end{proof}
	
	\begin{remark}
		We consider the logarithmic and the polynomial smoothness separately. 
		
		First, in case of logarithmic smoothness (i.e., index function $\varphi$ with $\nu = 0$ and $\tau > 0$), the convergence rate in \cref{t:sip} is in fact of order $(-\log \sigma^2)^{-2\tau}$, which is minimax optimal (cf.\ \cite{Pin80}). As in the mildly ill-posed problems, the ordered filters with the SOLIT rule perform as well as if the exact smoothness specified by $\tau$ were known.
		
		Second, in case of polynomial smoothness (i.e., index function $\varphi$ with $\nu > 0$ and $\tau \in \R$), the minimax optimal rate on $\mathcal{W}_{\varphi} (\rho)$ is of order 
		\begin{equation*}
			(\sigma^2)^{\frac{2\nu}{2\nu + 1}}(-\log\sigma^2)^{\frac{2\nu/\vartheta - 2\tau}{2\nu + 1}-\frac{2\nu}{2\nu+1}}.
		\end{equation*}
		The convergence rate in \cref{t:sip} is slower by a factor of $(-\log\sigma^2)^{{2\nu}/{(2\nu+1)}}$, and thus not minimax optimal. However, it is unclear whether our slower rate is improvable or not, as it is required to hold simultaneously over more than one smoothness class. In \cite{Tsy00}, it is shown that the optimal adaption rate (cf.~\cref{s:moa}) is of order
		$$
		(\sigma^2)^{\frac{2\nu}{2\nu + 1}}(-\log\sigma^2)^{\frac{2\nu/\vartheta - 2\tau}{2\nu + 1}},
		$$
		in the special case of $\vartheta = 1$. This is exactly the same as in \cref{t:sip}, which requires though $\vartheta \in (0,1/2]$. We thus conjecture that the convergence rate of the SOLIT rule would be adaptively minimax optimal over the whole range of smoothness classes specified by the index function $\varphi$ with $\nu > 0$ and $\tau\in \R$.
	\end{remark}
	
	\begin{remark}[Highly serious ill-posedness, $\vartheta > 1/2$]
		There is a restriction on the decaying speed of eigenvalues in \cref{t:sip}, more precisely, $\vartheta \le 1/2$. The reason behind is the sub-optimality of the upper bound on the variance part in \cref{l:bvd}. If a sharp bound were available, the result of \cref{t:sip} would hold for $\vartheta \le 1$. 
		
		Moreover, in the super fast decaying regime of $\vartheta > 1$, it turns out that the discretization of regularization parameter $\alpha$ uniform in geometric scaling, detailed in \eref{e:gsp:a}--\eref{e:gsp:c}, is no longer appropriate, as mentioned above. In fact, one should consider instead a smaller set $\mathcal{A}\equiv\mathcal{A}_{\theta_1,\theta_2}  = \{\alpha_m ;m\in \N_0\}$ such that 
		\begin{equation*}
			\infty > \theta_2 \ge \frac{\exp\bigl( (\log \vap_m)^{1/\vartheta}\bigr)}{\exp\bigl((\log \vap_{m-1})^{1/\vartheta}\bigr)}\ge \theta_1 > 1\qquad \mbox{for all}\quad m. 
		\end{equation*}
		Under such a discretization, it seems to be possible to show that the convergence rate of the SOLIT rule is adaptively minimax optimal for every $\vartheta > 0$. This would require a generalization of \cref{th:sw19}. Thus, the details are left as part of future research. 
	\end{remark}
	
	\section{Implementation and numerical simulations}\label{sec:implementation}
	
	To investigate the performance of SOLIT in practice, we start with a potential implementation. {\rev Note that for suitable filters $\flt_\alpha(\cdot)$, it is possible to implement the computation of $\hat f_\alpha$ without using the SVD of $T$. As an exemplary case we consider the Tikhonov regularization. Here $\hat f_\alpha$ can be computed as the solution of
	\[
	\left(\alpha \mathrm{Id} + T^*T\right) \hat f_\alpha = T^*Y,
	\]
    where $\mathrm{Id}$ denotes the identity operator on $\mathcal{X}$. 
	Similar computational possibilities exist e.g.\ for Landweber and Showalter regularizations. Given an SVD-free method to compute $\hat f_\alpha$, it is also possible to compute the SOLIT parameter $\hat \alpha$ without making use of the SVD, as we will detail below.}

	\subsection{Choosing the candidate parameters}
	
	To implement SOLIT, we first have to generate a sequence of regularization parameters $\alpha_0 > \alpha_1 > \cdots > \alpha_{m_{\max}} > 0$ obeying \eref{e:gsp:b}. This can be done using a line search for the function $V\left(\alpha\right) := \trace \bigl(\svar{\hat f_{\alpha}}\bigr)$. By recursively splitting an interval into halves, this algorithm provides an approximation $\hat \alpha$ to the solution $\alpha$ of $V  \left(\alpha\right) = \lambda$ with $\left|V\left(\alpha\right)- V\left(\hat \alpha\right)\right| \leq \epsilon$ for a prescribed error level $\epsilon$ in $\mathcal O\bigl(\log (1/\epsilon)\bigr)$ steps. If no such $\hat \alpha$ exists, an alternative value $\hat \alpha$ and the corresponding error $\epsilon$ is returned. 
	
	Now we fix a tuning parameter $\theta > 1$, set $\theta_2 =  \theta + \frac{\theta-1}{2}, \theta_1 =\theta- \frac{\theta-1}{2}$ and choose $\alpha_0$ as the line search approximation of $V\left(\alpha_0\right) = \sigma^2$ with $\epsilon = \frac{\theta-1}{2}$. If no such approximation exists, then we take the smallest possible value $\alpha_0$ such that $V\left(\alpha_0\right) > \sigma^2$ and enlarge $\theta_2$ correspondingly. Given $\alpha_{m-1}$, the relation \eref{e:gsp:b} corresponds to
	\[
	\log\left(\theta_1\right) + \log\left(V\left(\alpha_{m-1}\right)\right) \leq \log\left(V\left(\alpha_{m}\right)\right)  \leq \log\left(\theta_2\right) + \log\left(V\left(\alpha_{m-1}\right)\right),
	\]
	i.e., we can choose $\alpha_m$ as the line search approximation to the equation $\log\left(V\left(\alpha_m\right)\right) =\log\left(\theta\right) + \log\left(V\left(\alpha_{m-1}\right)\right)$ with $\epsilon = \frac{\theta-1}{2}$. Again, if no such solution exists we take a potentially larger $\alpha_m$ and enlarge $\theta_2$. 
	
	Note that the above procedure is in fact just a discrete approximation of the technique used in the proof of Lemma \ref{lem:alpha}. However, it also works in case of discontinuous filters $\alpha \mapsto \flt_\alpha(\lambda)$ as the sequence needs to be constructed only until a minimal value $\alpha_{m_{\max}} > 0$. From this point of view, the upper bound in \eref{e:gsp:b} is trivial (as only finitely many $\alpha$s are considered), but the above procedure ensures a good (i.e., small) value of $\theta_2$. 
	
	{\rev Since $\svar{\hat f_\alpha} = \sigma^2 T^*T \flt_\alpha(T^*T)^2$, we obtain in case of Tikhonov regularization that $V(\alpha) = \sigma^2\trace \bigl(\left(\alpha \mathrm{Id} + T^*T\right)^{-2} T^*T\bigr)$. To compute this function exactly, one has to evaluate  $T^*T$ on an arbitrary orthonormal system of $\X$, which numerically is as expensive as computing the SVD itself. However, by making use of random trace estimation algorithms \cite{at11}, the function $V$ can be approximated with high accuracy without computing the SVD and with reasonable computational cost.}
	
	\subsection{Computing the critical values $z_{m,m'}$}\label{sec:z}
	
	Given the regularization parameters $\alpha_m$, the implementation of SOLIT only requires thresholds in \eref{e:limit:b}, which are given in terms of the critical values $z_{m_1, m_2}(x)$ in \eref{e:limit:c}. The latter correspond to quantiles of the random variable
	\[
	\nop_{m_1,m_2}^2 = \norm{\err_{\alpha_{m_1}} - \err_{\alpha_{m_2}}}^2 = \sum_{k=1}^\infty \lambda_k \left(q_{\alpha_{m_1}} \left(\lambda_k\right) - q_{\alpha_{m_2}} \left(\lambda_k\right)\right)^2 \varepsilon_k^2
	\]
	with $\varepsilon_k \stackrel{\mathrm{i.i.d.}}{\sim} \mathcal N \left(0,1\right)$, i.e., standard Gaussian random variables.  In our implementation, the operator $\fop$ is approximated by a matrix $A$, and hence the sum is truncated at $k = n$. This implies that $\nop_{m_1,m_2}$ has the same distribution as
	\[
	Z_{m_1,m_2} := \varepsilon^\top A_{m_1,m_2} \varepsilon
	\]
	with the matrix $A_{m_1,m_2} = A^*A \bigl((q_{\alpha_{m_1}}(A^*A) - q_{\alpha_{m_2}}(A^*A)\bigr)^2$ and a random vector $\varepsilon \sim \mathcal N\left(0, I_n\right)$, i.e., $n$-dimensional standard Gaussian random vector. Note that $Z_{m_1,m_2}$ is a generalized $\chi^2$ random variable, the distribution of which can either be simulated by a Monte Carlo method or be approximated by a non-central $\chi^2$ approximation as described in \cite{ltz09}. There it is shown that the distribution of $\chi^2_l\left(\delta\right)$ (where $l$ is the number of degrees of freedom and $\delta$ the non-centrality parameter) is a reasonable approximation of the distribution of $Z_{m_1, m_2}$ in the following sense. With $c_k = \trace\left(A_{m_1,m_2}^k\right)$ for $k=1,2,3,4$ and the values $s_1 = {c_3}/{c_2^{3/2}}$ and $s_2 = {c_4}/{c_2^2}$, which correspond multiples of the skewness and kurtosis of $Z_{m_1,m_2}$, we define
	\begin{description}
		\item[if $s_1^2 > s_2$] $a = \frac{1}{s_1-\sqrt{s_1^2-s_2}}$, $\delta = s_1a^3 - a^2$, $l = a^2-2\delta$, and
		\item[if $s_1^2 \leq s_2$] $\delta = 0$, $l = \frac{c_2^3}{c_3^2}$, $a = \frac{1}{s_1}$.
	\end{description}
	Then it holds
	\[
	\Prob{Z_{m_1,m_2} > t} \approx \Prob{\chi^2_l \left(\delta\right) > \sqrt{2}a\frac{t-c_1}{\sqrt{2c_2}} + l + \delta}.
	\]
	We compared this approximation with the exact distribution of $Z$ (as computed in \cite{dg21}) and Monte Carlo simulations and found that the error in the quantiles is overwhelmingly small, and hence we use this approximation in our implementation.
	{\rev The necessary traces can once more be computed efficiently and without using the SVD of $T$ by means of (random) approximations, see \cite{at11}.} 
	
	\subsection{Numerical simulations}\label{sec:numerics}
	
	In this section we will report on numerical experiments supporting the minimax behaviour of SOLIT. We therefore follow the simulation of convergence rates in \cite{w18} and consider the same synthetic problems and there:
	\begin{description}
		\item[Antiderivative.] A mildly ill-posed problem with Fredholm integral operator $T : \mathbf L^2 \left(\left[0,1\right]\right) \to \mathbf L^2 \left(\left[0,1\right]\right)$ such that $\left(Tf\right)'' = -f$ for all $f \in \mathbf L^2 \left( \left[0,1\right] \right)$, see also \cite{hw14} for details. It is known that the eigenvalues $\lambda_k$ of $T^*T$ have the decay behaviour $\lambda_k \asymp k^{-4}$. 
		\item[2D-Gradiometry.] A severely ill-posed problem, which occurs when one aims to reconstruct the earth's gravitational potential on the surface from satellite measurements of the potentials second derivative in radial direction. In the two-dimensional case, using Fourier-coefficients $\hat f(k)$, the explicit formula is
		\[
		\left(Tf\right)\left(x\right) = \sum\limits_{k = -\infty}^{\infty}\left|k\right| \left(\left|k\right|+1\right) R^{-\left|k\right|-2} \exp\left(\textup{i} kx\right) \hat f\left(k\right)
		\]
		for the corresponding operator $T : \mathbf L^2 \left(S, \mu\right) \to \mathbf L^2 \left(R S, \mu\right)$ with the surface measure $\mu$ on the unit sphere $S = \set{x \in \mathbb R^2}{\norm{x}_{2} =1 }$ and the relative radius $R > 1$ of the satellite is known, see, e.g., \cite{kwh16}. Thus, the eigenvalues $\lambda_k$ of $T^*T$ are given as $\lambda_k = \left|k\right|^2 \left(\left|k\right|+1\right)^2 R^{-2\left|k\right|-4}$. Furthermore it follows that the Fourier coefficient $\hat f (0)$ cannot be determined from $Tf$. 
		\item[Heat.] A severely ill-posed problem where one aims at reconstructing the source $f$ in the periodic heat equation
		\begin{eqnarray*}
			\frac{\partial u}{\partial t} \left(x,t\right) = \frac{\partial^2 u}{\partial t^2} \left(x,t\right) & \mathrm{in}\quad\left(-\pi,\pi\right] \times \left(0,\bar t\right), \\ u\left(x,0\right) = f(x) & \mathrm{on}\quad \left[-\pi,\pi\right], \\ u\left(-\pi,t\right) = u\left(\pi,t\right) & \mathrm{on}\quad t \in \left(0,\bar t\right]
		\end{eqnarray*}
		from measurements of the final heat distribution $g = u\left(\cdot,\bar t\right) $ with $\bar t>0$. By separation of variables, an explicit representation of $T: \mathbf L^2 \left(\left[-\pi,\pi\right]\right) \to \mathbf L^2 \left(\left[-\pi,\pi\right]\right)$ in form of a Fourier series can be computed, i.e.,
		\[
		\left(Tf\right)\left(x\right) = \sum\limits_{k = -\infty}^{\infty} \exp\left(-k^2\bar t\right) \exp\left(\textup{i} kx\right) \tilde f\left(k\right),
		\]
		where $\tilde f\left(k\right)$ is the $k$-th Fourier coefficient of $f$. 
		Consequently, the eigenvalues $\lambda_k$ of $T^*T$ are given by $\lambda_k = \exp\left(-2k^2\bar t\right)$. 
	\end{description}
	We simulate the rate of convergence for specific exact solutions $f^\dagger$ described below when using three different ordered filters, namely,
	\begin{itemize}
		\item \textit{Tikhonov regularization}, where $q_\alpha \left(\lambda\right) = \frac{1}{\lambda+\alpha}$. In this case, all requirements of Definition \ref{d:of} are satisfied with $C_q' = 1$, and the maximal qualification is $\varphi(x) = x$. Moreover, the mapping $\alpha \mapsto q_\alpha(\lambda)$ is obviously continuous and hence Lemma \ref{lem:alpha} is applicable.
		\item \textit{Showalter's method}, where $q_\alpha \left(\lambda\right) = \frac{1-\exp\left(-{\lambda}/{\alpha}\right)}{\lambda}$. Again, all requirements of Definition \ref{d:of} are met with $C_q' = 1$ and an arbitrary index function $\varphi$. Moreover, the mapping $\alpha \mapsto q_\alpha(\lambda)$ is obviously continuous and hence Lemma \ref{lem:alpha} is applicable.
		\item \textit{Spectral cut-off regularization}, where $q_\alpha \left(\lambda\right) = \frac{1}{\lambda} \mathbf 1_{[\alpha, \infty)}(\lambda)$. Again, all requirements of Definition \ref{d:of} are met with $C_q' = 1$ and an arbitrary index function $\varphi$, but in this case \cref{lem:alpha} is not applicable and as discussed in \cref{ss:tc} the choice of suitable candidates $\alpha$'s is questionable.
	\end{itemize}
	Besides the rate of convergence, we also depict the \textit{price of adaptation}
	\begin{equation}\label{eq:poa}
		\left(\sqrt{\mathcal R_{m_*}} + C_{2,m_*}\right)^2
	\end{equation}
	on the right-hand side of the oracle inequality in Theorem \ref{t:orc}.
	
	\subsubsection{The antiderivative problem}\label{sec:antiderivative}
	
	We choose the continuous function
	\begin{equation}\label{eq:exact_sol}
		f^\dagger\left(x\right) = x \mathbf 1_{\left[0,\frac12\right]}(x) + (1-x)\mathbf 1_{\left[\frac12,1\right]}(x)
	\end{equation}
	as exact solution. As argued in \cite{LiWe20}, the optimal rate of convergence in this situation is $\mathcal O \left(\sigma^{3/4-\varepsilon}\right)$ for any $\varepsilon>0$. To avoid an inverse crime, the exact data $g = Tf^\dagger$ is implemented analytically:
	\[
	g\left(x\right) = -\frac{x\left(4x^2-3\right)}{24}\mathbf 1_{\left[0,\frac12\right]}(x)+ \frac{\left(x-1\right)\left(4x^2-8x+1\right)}{24}\mathbf 1_{\left[\frac12,1\right]}(x).
	\]
	
	In Figure \ref{fig:antiderivative} we depict the ideal rate $\mathcal O \left(\sigma^{3/4}\right)$ the convergence rate obtained with SOLIT, the convergence rate obtained by the oracle $\alpha_*$ in \eref{e:orcr}, and the convergence rate obtained for the optimal choice
	\begin{equation}\label{eq:alpha_opt}
		\alpha_{\mathrm{opt}} := \argmin_{\alpha \in \left\{\alpha_0,...,\alpha_{m_{\max}}\right\}} \norm{f_\alpha -f^\dagger}_{\X}.
	\end{equation}
	The results show that SOLIT clearly obtains the minimax rate and yields also a surprisingly small loss compared to the oracle choice. For Tikhonov and Showalter regularization, the oracle choice and the optimal choice cannot even be distinguished, for spectral cut-off regularization (where determining the candidate $\alpha$'s is more subtle) a less regular behavior can be observed. It is also visible that the price of adaptation \eref{eq:poa} is of the same order as the optimal rate, supporting our analysis.
	
	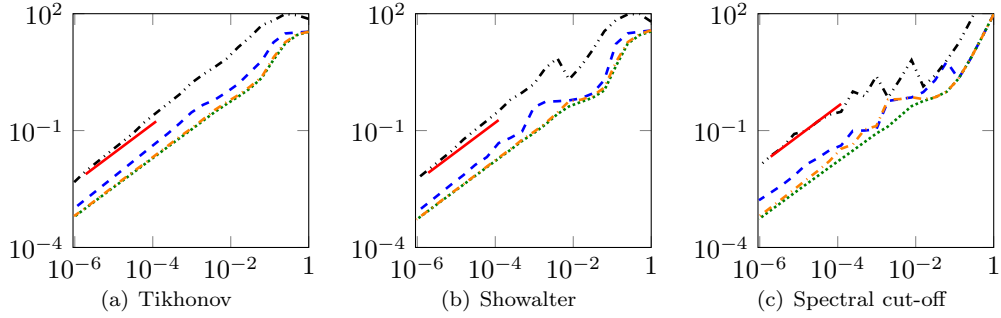
\begin{figure}[!htb]
		\setlength{\fwidth}{3.1cm}
		\setlength{\fheight}{3.1cm}
		\subfigure[Tikhonov]{\begin{tikzpicture}[baseline]

\begin{axis}[%
width=\fwidth,
height=\fheight,
scale only axis,
xmode=log,
xmin=9e-07,
xmax=1,
xminorticks=false,
xtick={0.000001,0.0001,0.01,1},
xticklabels = {$10^{-6}$,$10^{-4}$,$10^{-2}$,$1$},
ymode=log,
ymin=1e-04,
ymax=100,
yminorticks=false
]
\addplot [color=blue,dashed, line width = 0.35mm]
  table[row sep=crcr]{%
1	34.4346052008019\\
0.5	33.00044799331\\
0.25	31.1037119181398\\
0.125	18.4382912235804\\
0.0625	6.1727196294087\\
0.03125	3.13428553576885\\
0.015625	1.58382227787991\\
0.0078125	0.964567291334749\\
0.00390625	0.619485565244678\\
0.001953125	0.458304391925889\\
0.0009765625	0.29123005276499\\
0.00048828125	0.152842089874333\\
0.000244140625	0.0840496867213426\\
0.0001220703125	0.0485263516328282\\
6.103515625e-05	0.0278240726678024\\
3.0517578125e-05	0.0159028115399263\\
1.52587890625e-05	0.00913147537682874\\
7.62939453125e-06	0.00523494612962139\\
3.814697265625e-06	0.00300792605212834\\
1.9073486328125e-06	0.00173056403984703\\
9.5367431640625e-07	0.000996922982086576\\
};

\addplot [color=black!50!green,densely dotted, line width = 0.35mm]
  table[row sep=crcr]{%
1	34.4346052008019\\
0.5	28.2995735437004\\
0.25	17.1035458219707\\
0.125	6.55922357659464\\
0.0625	2.25214910188325\\
0.03125	1.20884460567133\\
0.015625	0.743958293397229\\
0.0078125	0.478883790005038\\
0.00390625	0.300436037664164\\
0.001953125	0.178954106778099\\
0.0009765625	0.105986376884406\\
0.00048828125	0.0639416968962967\\
0.000244140625	0.0381666864240348\\
0.0001220703125	0.0228652197142978\\
6.103515625e-05	0.0136092325486725\\
3.0517578125e-05	0.00808570447069591\\
1.52587890625e-05	0.00482873397015868\\
7.62939453125e-06	0.00288822189482184\\
3.814697265625e-06	0.00171970800452188\\
1.9073486328125e-06	0.00102434082891929\\
9.5367431640625e-07	0.000606676051465466\\
};

\addplot [color=orange,dash dot, line width = 0.35mm]
  table[row sep=crcr]{%
1	34.4346052008019\\
0.5	29.5908294158034\\
0.25	18.566526448147\\
0.125	7.54714152318023\\
0.0625	2.39479144831643\\
0.03125	1.34512732677948\\
0.015625	0.81652422810714\\
0.0078125	0.533999266270759\\
0.00390625	0.339548023374774\\
0.001953125	0.194777764067603\\
0.0009765625	0.113243342970829\\
0.00048828125	0.0695020141280952\\
0.000244140625	0.040933186543236\\
0.0001220703125	0.0241620949470407\\
6.103515625e-05	0.0142270008792885\\
3.0517578125e-05	0.00866086678844794\\
1.52587890625e-05	0.00510828688031534\\
7.62939453125e-06	0.00301888394778842\\
3.814697265625e-06	0.00177971509857883\\
1.9073486328125e-06	0.00105090789346874\\
9.5367431640625e-07	0.000637154656833028\\
};

\addplot [color=black,dash dot dot, line width = 0.35mm]
  table[row sep=crcr]{%
1	73.486163550366\\
0.5	91.5482345800251\\
0.25	98.2821799849874\\
0.125	67.4881798331964\\
0.0625	50.3864187390774\\
0.03125	25.1669285298021\\
0.015625	13.8897179063655\\
0.0078125	6.98074157453245\\
0.00390625	4.29241479871801\\
0.001953125	2.74843440392168\\
0.0009765625	1.66177582345191\\
0.00048828125	0.851323056697261\\
0.000244140625	0.491351977896981\\
0.0001220703125	0.284378820848071\\
6.103515625e-05	0.163123579347897\\
3.0517578125e-05	0.0833502999757074\\
1.52587890625e-05	0.0479348292067059\\
7.62939453125e-06	0.0276379359777081\\
3.814697265625e-06	0.015843096284218\\
1.9073486328125e-06	0.00910706960456529\\
9.5367431640625e-07	0.00469611814933546\\
};

\addplot [color=red,solid, line width = 0.35mm]
  table[row sep=crcr]{%
0.0001220703125	0.17217302600401\\
6.103515625e-05	0.102374693767758\\
3.0517578125e-05	0.0608723571124215\\
1.52587890625e-05	0.0361949200925391\\
7.62939453125e-06	0.0215216282505012\\
3.814697265625e-06	0.0127968367209698\\
1.9073486328125e-06	0.00760904463905269\\
};
\end{axis}
\end{tikzpicture}}
		\subfigure[Showalter]{\begin{tikzpicture}[baseline]

\begin{axis}[%
width=\fwidth,
height=\fheight,
scale only axis,
xmode=log,
xmin=9e-07,
xmax=1,
xminorticks=false,
xtick={0.000001,0.0001,0.01,1},
xticklabels = {$10^{-6}$,$10^{-4}$,$10^{-2}$,$1$},
ymode=log,
ymin=1e-04,
ymax=100,
yminorticks=false
]
\addplot [color=blue,dashed, line width = 0.35mm]
  table[row sep=crcr]{%
1	37.1597767877455\\
0.5	34.3333392706925\\
0.25	31.2970684628132\\
0.125	16.8719278351276\\
0.0625	1.98715740842792\\
0.03125	0.892240081588146\\
0.015625	0.67185462484027\\
0.0078125	0.594749819081486\\
0.00390625	0.56660392475253\\
0.001953125	0.551728687021033\\
0.0009765625	0.394609687130762\\
0.00048828125	0.0950088077749352\\
0.000244140625	0.0648926826750929\\
0.0001220703125	0.0478120472588625\\
6.103515625e-05	0.0208966699837359\\
3.0517578125e-05	0.0131597148717988\\
1.52587890625e-05	0.0071901108069144\\
7.62939453125e-06	0.00421867051218656\\
3.814697265625e-06	0.0024150392722145\\
1.9073486328125e-06	0.00138671877086013\\
9.5367431640625e-07	0.000801071597942116\\
};
\label{sho_antider_solit}

\addplot [color=black!50!green,densely dotted, line width = 0.35mm]
  table[row sep=crcr]{%
1	35.2841455444372\\
0.5	28.1635591725754\\
0.25	17.0285755808228\\
0.125	4.56152329135715\\
0.0625	1.07320297917782\\
0.03125	0.695689511387595\\
0.015625	0.541170821960925\\
0.0078125	0.417016082122283\\
0.00390625	0.26371776390154\\
0.001953125	0.144079569764039\\
0.0009765625	0.0859494765755051\\
0.00048828125	0.0543254893691252\\
0.000244140625	0.0318975470292153\\
0.0001220703125	0.0191915048164184\\
6.103515625e-05	0.0115059660962271\\
3.0517578125e-05	0.0068640811297351\\
1.52587890625e-05	0.00409471918573135\\
7.62939453125e-06	0.00244423051317736\\
3.814697265625e-06	0.00146121022941708\\
1.9073486328125e-06	0.00086879544847513\\
9.5367431640625e-07	0.000518179276199285\\
};
\label{sho_antider_opt}

\addplot [color=orange,dash dot, line width = 0.35mm]
  table[row sep=crcr]{%
1	37.7127269719132\\
0.5	29.5233262034195\\
0.25	18.6211834310358\\
0.125	5.7479161284786\\
0.0625	1.26588539785967\\
0.03125	0.771095925004973\\
0.015625	0.618710871410551\\
0.0078125	0.567479623739951\\
0.00390625	0.28853223941684\\
0.001953125	0.156280394376741\\
0.0009765625	0.0925202303085975\\
0.00048828125	0.0583404717961207\\
0.000244140625	0.0334092975260818\\
0.0001220703125	0.0200427689418476\\
6.103515625e-05	0.0119500528466579\\
3.0517578125e-05	0.0070656151806021\\
1.52587890625e-05	0.00418845554265169\\
7.62939453125e-06	0.00251448763718355\\
3.814697265625e-06	0.00149280213699389\\
1.9073486328125e-06	0.000882215049311644\\
9.5367431640625e-07	0.000524194014681729\\
};
\label{sho_antider_oracle}

\addplot [color=black,dash dot dot, line width = 0.35mm]
  table[row sep=crcr]{%
1	62.4315265271281\\
0.5	98.958352239738\\
0.25	96.1228028193589\\
0.125	68.5384370638214\\
0.0625	35.3422590321417\\
0.03125	11.6063375127261\\
0.015625	4.4200326403567\\
0.0078125	2.05467956714951\\
0.00390625	7.20090803288787\\
0.001953125	4.04806216574354\\
0.0009765625	1.55885620330392\\
0.00048828125	0.944602603742327\\
0.000244140625	0.614761899057554\\
0.0001220703125	0.315503031319444\\
6.103515625e-05	0.179570403511914\\
3.0517578125e-05	0.103606376605395\\
1.52587890625e-05	0.0590178219451205\\
7.62939453125e-06	0.0300833780197447\\
3.814697265625e-06	0.01722043727485\\
1.9073486328125e-06	0.00983682495115605\\
9.5367431640625e-07	0.00562973656855303\\
};
\label{sho_antider_poa}

\addplot [color=red,solid, line width = 0.35mm]
  table[row sep=crcr]{%
0.0001220703125	0.185798883938728\\
6.103515625e-05	0.11047667736975\\
3.0517578125e-05	0.0656898253849833\\
1.52587890625e-05	0.0390594038655543\\
7.62939453125e-06	0.023224860492341\\
3.814697265625e-06	0.0138095846712187\\
1.9073486328125e-06	0.00821122817312291\\
};
\label{sho_antider_orc}

\end{axis}
\end{tikzpicture}}
		\subfigure[Spectral cut-off]{\begin{tikzpicture}[baseline]

\begin{axis}[%
width=\fwidth,
height=\fheight,
scale only axis,
xmode=log,
xmin=9e-07,
xmax=1,
xminorticks=false,
xtick={0.000001,0.0001,0.01,1},
xticklabels = {$10^{-6}$,$10^{-4}$,$10^{-2}$,$1$},
ymode=log,
ymin=1e-04,
ymax=100,
yminorticks=false
]
\addplot [color=blue,dashed, line width = 0.35mm]
  table[row sep=crcr]{%
1	96.7613556194184\\
0.5	24.8213785979956\\
0.25	6.75844199469257\\
0.125	2.14393315909152\\
0.0625	5.86081615772723\\
0.03125	1.95266595408502\\
0.015625	0.970208837464283\\
0.0078125	0.709121830134616\\
0.00390625	0.641997038163557\\
0.001953125	0.575510773982861\\
0.0009765625	0.107859930567756\\
0.00048828125	0.100338157091682\\
0.000244140625	0.098635401557103\\
0.0001220703125	0.0414327534164545\\
6.103515625e-05	0.0311402650875834\\
3.0517578125e-05	0.0195493991359009\\
1.52587890625e-05	0.0133545461504953\\
7.62939453125e-06	0.00695213284867098\\
3.814697265625e-06	0.00406321550569711\\
1.9073486328125e-06	0.00255405955815617\\
9.5367431640625e-07	0.00163181533910967\\
};

\addplot [color=black!50!green,densely dotted, line width = 0.35mm]
  table[row sep=crcr]{%
1	96.7613556194184\\
0.5	24.8213785979956\\
0.25	6.75844199469257\\
0.125	2.14393315909152\\
0.0625	0.997051460096424\\
0.03125	0.712843885188594\\
0.015625	0.573558199311297\\
0.0078125	0.422586826401473\\
0.00390625	0.232003364823462\\
0.001953125	0.127786938123446\\
0.0009765625	0.0840702984029624\\
0.00048828125	0.0507816046812995\\
0.000244140625	0.0306456845114302\\
0.0001220703125	0.0186367550809227\\
6.103515625e-05	0.0113434729507071\\
3.0517578125e-05	0.00682000132755168\\
1.52587890625e-05	0.0041314637264253\\
7.62939453125e-06	0.00248736022550163\\
3.814697265625e-06	0.00149526163399991\\
1.9073486328125e-06	0.000897895803775084\\
9.5367431640625e-07	0.000537519376790159\\
};

\addplot [color=orange,dash dot, line width = 0.35mm]
  table[row sep=crcr]{%
1	96.7613556194184\\
0.5	24.8213785979956\\
0.25	6.75844199469257\\
0.125	2.14393315909152\\
0.0625	0.997051460096424\\
0.03125	0.712843885188594\\
0.015625	0.640733281439471\\
0.0078125	0.719843221232393\\
0.00390625	0.642462541664692\\
0.001953125	0.623415638031125\\
0.0009765625	0.130704581632681\\
0.00048828125	0.106267469066804\\
0.000244140625	0.0441205008764599\\
0.0001220703125	0.034092512270235\\
6.103515625e-05	0.0164782299780303\\
3.0517578125e-05	0.00916619588033719\\
1.52587890625e-05	0.00538264099881863\\
7.62939453125e-06	0.00308750365944124\\
3.814697265625e-06	0.0018028806927148\\
1.9073486328125e-06	0.00105886505901848\\
9.5367431640625e-07	0.00062379080862652\\
};

\addplot [color=black,dash dot dot, line width = 0.35mm]
  table[row sep=crcr]{%
1	981.371100162451\\
0.5	247.343558458185\\
0.25	63.5564096454467\\
0.125	17.2264934228083\\
0.0625	6.3995319704255\\
0.03125	2.59178869129253\\
0.015625	1.36499102176177\\
0.0078125	6.43701653642637\\
0.00390625	2.04691578033208\\
0.001953125	0.699393889900199\\
0.0009765625	2.491803278676\\
0.00048828125	0.761179075957973\\
0.000244140625	1.00390669663867\\
0.0001220703125	0.304038052500018\\
6.103515625e-05	0.256195375212989\\
3.0517578125e-05	0.170605588863469\\
1.52587890625e-05	0.0980113768609839\\
7.62939453125e-06	0.0818656428711169\\
3.814697265625e-06	0.0369326979707562\\
1.9073486328125e-06	0.0224090205045232\\
9.5367431640625e-07	0.0120532635538813\\
};

\addplot [color=red,solid, line width = 0.35mm]
  table[row sep=crcr]{%
0.0001220703125	0.483806778097092\\
6.103515625e-05	0.287673231399802\\
3.0517578125e-05	0.171051526788235\\
1.52587890625e-05	0.101707846344324\\
7.62939453125e-06	0.0604758472621365\\
3.814697265625e-06	0.0359591539249753\\
1.9073486328125e-06	0.0213814408485293\\
};
\end{axis}
\end{tikzpicture}}
		\caption{Simulation results for the antiderivate problem. Depicted are different noise levels $\sigma$ (x-axis) against empirical mean squared errors $\E{\left\Vert \hat f_\alpha - f\right\Vert_{\X}^2}$ simulated in $M = 10^4$ Monte Carlo  runs, for different parameter choices $\alpha$, namely SOLIT (eq.~\ref{e:limit:a}--\ref{e:limit:d}, \ref{sho_antider_solit}), optimal alpha (eq.~\ref{eq:alpha_opt}, \ref{sho_antider_opt}) and the oracle choice (eq.~\ref{e:orcr}, \ref{sho_antider_oracle}). Furthermore shown is the empirical price of adaptation (eq.~\ref{eq:poa}, \ref{sho_antider_poa}), and a slope (\ref{sho_antider_orc}) indicating the optimal rate of convergence $\mathcal O \left(\sigma^{3/4}\right)$.}
		\label{fig:antiderivative}
	\end{figure}
	
	\subsubsection{The 2D-gradiometry problem}\label{sec:gradiometry}
	
	As $\hat f \left(0\right)$ cannot be determined from $Tf^\dagger$, we adopt the exact solution from Section \ref{sec:antiderivative} to
	\[
	f^\dagger \left(x\right) = \frac{\pi}{2} - \left|x\right|, \qquad x \in \left[-\pi,\pi\right].
	\]
	As discussed in \cite{w18}, the optimal rate of convergence in this situation is $\mathcal O \bigl(\left(-\log\sigma\right)^{-3+\varepsilon}\bigr)$ for any $\varepsilon>0$. Again, the ideal data $g$ can be computed analytically as,
	\[
	g\left(x\right) = \frac{4}{\pi} \sum\limits_{m \in \mathbb N} \left(1 + \frac{1}{2m+1}\right) R^{-2m-3} \cos\bigl(\left(2m+1\right) x\bigr), \quad x \in \left[-\pi,\pi\right]
	\]
	and in our implementation we truncate this sum at $m = 128$. In our simulations, we set $R = 2$.
	
	In Figure \ref{fig:gradiometry} we again depict besides the ideal rate $\mathcal O \bigl(\left(-\log\sigma\right)^{-3}\bigr)$ the convergence rates obtained by SOLIT, \eref{e:orcr} and \eref{eq:alpha_opt} as well as empirical price of adaptation \eref{eq:poa}. Despite the faster decay of the singular values it seems that SOLIT still obtains the minimax rate and has again a surprisingly small loss compared to the oracle choice. Furthermore, the price of adaptation \eref{eq:poa} is also of the same order as the optimal rate.
	
	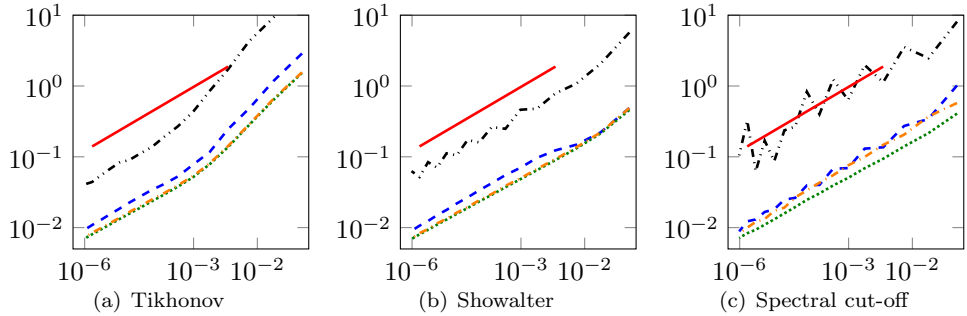
\begin{figure}[!htb]
		\setlength{\fwidth}{3.1cm}
		\setlength{\fheight}{3.1cm}
		\subfigure[Tikhonov]{\begin{tikzpicture}[baseline]

\begin{axis}[%
width=\fwidth,
height=\fheight,
scale only axis,
xmin=1.2,
xmax=2.7,
xminorticks=false,
xtick={1.52718,1.93264,2.6258},
xticklabels = {$10^{-2}$,$10^{-3}$,$10^{-6}$},
x dir=reverse,
ymode=log,
ymin=0.005,
ymax=10,
yminorticks=false
]
\addplot [color=blue,dashed, line width = 0.35mm]
  table[row sep=crcr]{%
1.24292499185244	2.88047527512016\\
1.42524654864639	1.17024522140305\\
1.57939722847365	0.480624294654273\\
1.71292862109817	0.246023511577616\\
1.83071165675456	0.116541971235295\\
1.93607217241238	0.0775863185990559\\
2.03138235221671	0.0530367492084689\\
2.11839372920634	0.0414327924738315\\
2.19843643687987	0.0341089508401843\\
2.27254440903359	0.0275455752267175\\
2.34153728052055	0.0225842412534342\\
2.40607580165812	0.0184237768036898\\
2.46670042347455	0.0153636329960267\\
2.5238588373145	0.0127973977460144\\
2.57792605858478	0.0108924522976458\\
2.62921935297233	0.00924315078786608\\
};

\addplot [color=black!50!green,densely dotted, line width = 0.35mm]
  table[row sep=crcr]{%
1.24292499185244	1.53116930696429\\
1.42524654864639	0.633788457236928\\
1.57939722847365	0.281472531314574\\
1.71292862109817	0.137685872331763\\
1.83071165675456	0.0784409714526546\\
1.93607217241238	0.0515437728902835\\
2.03138235221671	0.0380688762053849\\
2.11839372920634	0.0295133103370171\\
2.19843643687987	0.0235613336031645\\
2.27254440903359	0.0190858149498867\\
2.34153728052055	0.0157605972046439\\
2.40607580165812	0.0131304143028097\\
2.46670042347455	0.011041804241173\\
2.5238588373145	0.00937839634056699\\
2.57792605858478	0.00805605052778433\\
2.62921935297233	0.0069403711294165\\
};

\addplot [color=orange,dash dot, line width = 0.35mm]
  table[row sep=crcr]{%
1.24292499185244	1.56021116833149\\
1.42524654864639	0.653805906431651\\
1.57939722847365	0.287293821517211\\
1.71292862109817	0.142466887475196\\
1.83071165675456	0.0807196151099047\\
1.93607217241238	0.0535876157112054\\
2.03138235221671	0.0397461504112482\\
2.11839372920634	0.0309233553827775\\
2.19843643687987	0.024936473028965\\
2.27254440903359	0.0201716483075279\\
2.34153728052055	0.016617602400011\\
2.40607580165812	0.0136848238727714\\
2.46670042347455	0.0115283821304709\\
2.5238588373145	0.0098307357589696\\
2.57792605858478	0.00846262791499174\\
2.62921935297233	0.00724549966809749\\
};

\addplot [color=black,dash dot dot, line width = 0.35mm]
  table[row sep=crcr]{%
1.24292499185244	23.7241269637586\\
1.42524654864639	9.20014621770182\\
1.57939722847365	4.28264885034806\\
1.71292862109817	1.72397636484312\\
1.83071165675456	0.849642388363285\\
1.93607217241238	0.431956667032361\\
2.03138235221671	0.270769335095993\\
2.11839372920634	0.192642949988305\\
2.19843643687987	0.137560790633677\\
2.27254440903359	0.111243267454617\\
2.34153728052055	0.0902146663885546\\
2.40607580165812	0.0827002182336957\\
2.46670042347455	0.0670464694779534\\
2.5238588373145	0.0543841686581816\\
2.57792605858478	0.0441750201324961\\
2.62921935297233	0.0405519794239357\\
};

\addplot [color=red,solid, line width = 0.35mm]
  table[row sep=crcr]{%
1.71292862109817	1.87796741803241\\
1.83071165675456	1.3189565405111\\
1.93607217241238	0.961519318032592\\
2.03138235221671	0.722403694990678\\
2.11839372920634	0.556434790528121\\
2.19843643687987	0.437651032331631\\
2.27254440903359	0.35040791473491\\
2.34153728052055	0.284894612750398\\
2.40607580165812	0.234745927254051\\
2.46670042347455	0.1957092037518\\
2.5238588373145	0.164869567563888\\
2.57792605858478	0.140183600821197\\
};

\end{axis}
\end{tikzpicture}}
		\subfigure[Showalter]{\begin{tikzpicture}[baseline]

\begin{axis}[%
width=\fwidth,
height=\fheight,
scale only axis,
xmin=1.2,
xmax=2.7,
xminorticks=false,
xtick={1.52718,1.93264,2.6258},
xticklabels = {$10^{-2}$,$10^{-3}$,$10^{-6}$},
x dir=reverse,
ymode=log,
ymin=0.005,
ymax=10,
yminorticks=false
]
\addplot [color=blue,dashed, line width = 0.35mm]
  table[row sep=crcr]{%
1.24292499185244	0.4881026693739\\
1.42524654864639	0.221873783435026\\
1.57939722847365	0.148507435503784\\
1.71292862109817	0.123361967284777\\
1.83071165675456	0.096628168652816\\
1.93607217241238	0.0700365052224398\\
2.03138235221671	0.0552472845314825\\
2.11839372920634	0.0424461145475312\\
2.19843643687987	0.0336477333832978\\
2.27254440903359	0.026689856016259\\
2.34153728052055	0.0219162886682253\\
2.40607580165812	0.0177900672898922\\
2.46670042347455	0.014967210676357\\
2.5238588373145	0.0124170422484179\\
2.57792605858478	0.0106527455435437\\
2.62921935297233	0.00901278981242804\\
};
\label{sho_grad_solit}

\addplot [color=black!50!green,densely dotted, line width = 0.35mm]
  table[row sep=crcr]{%
1.24292499185244	0.458408816487755\\
1.42524654864639	0.214324441110194\\
1.57939722847365	0.131740024081525\\
1.71292862109817	0.0899470478371944\\
1.83071165675456	0.064768634105063\\
1.93607217241238	0.0482933414928583\\
2.03138235221671	0.0371835036253854\\
2.11839372920634	0.0290587074944546\\
2.19843643687987	0.0233898376048875\\
2.27254440903359	0.0189049948483937\\
2.34153728052055	0.015633982074276\\
2.40607580165812	0.0130087510638357\\
2.46670042347455	0.0110149364892248\\
2.5238588373145	0.00932534890480299\\
2.57792605858478	0.0080418048401385\\
2.62921935297233	0.00691923136243039\\
};
\label{sho_grad_opt}

\addplot [color=orange,dash dot, line width = 0.35mm]
  table[row sep=crcr]{%
1.24292499185244	0.4881026693739\\
1.42524654864639	0.221873783435026\\
1.57939722847365	0.136560560933647\\
1.71292862109817	0.0931340268376578\\
1.83071165675456	0.0674459901695637\\
1.93607217241238	0.0498057812075558\\
2.03138235221671	0.0389690612270437\\
2.11839372920634	0.0299277770895598\\
2.19843643687987	0.0243537533672953\\
2.27254440903359	0.0194539915844364\\
2.34153728052055	0.0162709103483855\\
2.40607580165812	0.0133777869382659\\
2.46670042347455	0.0114825776541855\\
2.5238588373145	0.00956019025231214\\
2.57792605858478	0.00834947181019872\\
2.62921935297233	0.00707998387858204\\
};
\label{sho_grad_oracle}

\addplot [color=black,dash dot dot, line width = 0.35mm]
  table[row sep=crcr]{%
1.24292499185244	5.6801182111536\\
1.42524654864639	2.0725813186149\\
1.57939722847365	1.0825624824678\\
1.71292862109817	0.787673147263953\\
1.83071165675456	0.478997888821913\\
1.93607217241238	0.462068999979764\\
2.03138235221671	0.248563114354214\\
2.11839372920634	0.263387862840199\\
2.19843643687987	0.159894383155212\\
2.27254440903359	0.172456450254858\\
2.34153728052055	0.104930321826031\\
2.40607580165812	0.11211048579995\\
2.46670042347455	0.0703218321987714\\
2.5238588373145	0.0840572626307961\\
2.57792605858478	0.05196344833686\\
2.62921935297233	0.0627552990791438\\
};
\label{sho_grad_poa}

\addplot [color=red,solid, line width = 0.35mm]
  table[row sep=crcr]{%
1.71292862109817	1.87796741803241\\
1.83071165675456	1.3189565405111\\
1.93607217241238	0.961519318032592\\
2.03138235221671	0.722403694990678\\
2.11839372920634	0.556434790528121\\
2.19843643687987	0.437651032331631\\
2.27254440903359	0.35040791473491\\
2.34153728052055	0.284894612750398\\
2.40607580165812	0.234745927254051\\
2.46670042347455	0.1957092037518\\
2.5238588373145	0.164869567563888\\
2.57792605858478	0.140183600821197\\
};
\label{sho_grad_orc}

\end{axis}
\end{tikzpicture}}
		\subfigure[Spectral cut-off]{\begin{tikzpicture}[baseline]

\begin{axis}[%
width=\fwidth,
height=\fheight,
scale only axis,
xmin=1.2,
xmax=2.7,
xminorticks=false,
xtick={1.52718,1.93264,2.6258},
xticklabels = {$10^{-2}$,$10^{-3}$,$10^{-6}$},
x dir=reverse,
ymode=log,
ymin=0.005,
ymax=10,
yminorticks=false
]
\addplot [color=blue,dashed, line width = 0.35mm]
  table[row sep=crcr]{%
1.24292499185244	1.01804668781476\\
1.42524654864639	0.341443131760484\\
1.57939722847365	0.272599336893987\\
1.71292862109817	0.136291340313292\\
1.83071165675456	0.13047493101373\\
1.93607217241238	0.0692772286069503\\
2.03138235221671	0.0679900700282021\\
2.11839372920634	0.0401556524945947\\
2.19843643687987	0.0394523002505023\\
2.27254440903359	0.0263322318614562\\
2.34153728052055	0.024861944417086\\
2.40607580165812	0.0181267601968329\\
2.46670042347455	0.0166466575843846\\
2.5238588373145	0.0130522855637434\\
2.57792605858478	0.0122716025986789\\
2.62921935297233	0.00884954258595151\\
};

\addplot [color=black!50!green,densely dotted, line width = 0.35mm]
  table[row sep=crcr]{%
1.24292499185244	0.407598771775157\\
1.42524654864639	0.215572789913231\\
1.57939722847365	0.140357703143579\\
1.71292862109817	0.0939310194189754\\
1.83071165675456	0.0689170121624923\\
1.93607217241238	0.0503666567210548\\
2.03138235221671	0.039397321760474\\
2.11839372920634	0.0303974986432019\\
2.19843643687987	0.0246916490711975\\
2.27254440903359	0.0195211010693305\\
2.34153728052055	0.015907271905355\\
2.40607580165812	0.013101365528916\\
2.46670042347455	0.011012369976345\\
2.5238588373145	0.0093449642081225\\
2.57792605858478	0.00833569222907436\\
2.62921935297233	0.00714411447507227\\
};

\addplot [color=orange,dash dot, line width = 0.35mm]
  table[row sep=crcr]{%
1.24292499185244	0.575109453388027\\
1.42524654864639	0.370451754325367\\
1.57939722847365	0.205735241000009\\
1.71292862109817	0.147433570375452\\
1.83071165675456	0.0995274750371442\\
1.93607217241238	0.0755964148777115\\
2.03138235221671	0.0549053691492974\\
2.11839372920634	0.0430860435692723\\
2.19843643687987	0.0335243096410759\\
2.27254440903359	0.0267556293221573\\
2.34153728052055	0.0191789514095192\\
2.40607580165812	0.0172350801324341\\
2.46670042347455	0.0132426669929842\\
2.5238588373145	0.0120453430312694\\
2.57792605858478	0.0100011061425028\\
2.62921935297233	0.00834351435144194\\
};

\addplot [color=black,dash dot dot, line width = 0.35mm]
  table[row sep=crcr]{%
1.24292499185244	8.09847836942899\\
1.42524654864639	2.47084858301074\\
1.57939722847365	3.55128767388072\\
1.71292862109817	1.12885165534423\\
1.83071165675456	1.97792184642753\\
1.93607217241238	0.635465678983422\\
2.03138235221671	1.23767341716217\\
2.11839372920634	0.398990262576602\\
2.19843643687987	0.831188451134735\\
2.27254440903359	0.26747029867943\\
2.34153728052055	0.23653288890722\\
2.40607580165812	0.0868113214303467\\
2.46670042347455	0.169083828815685\\
2.5238588373145	0.0624212061322939\\
2.57792605858478	0.324877103054261\\
2.62921935297233	0.104921968179325\\
};

\addplot [color=red,solid, line width = 0.35mm]
  table[row sep=crcr]{%
1.71292862109817	1.87796741803241\\
1.83071165675456	1.3189565405111\\
1.93607217241238	0.961519318032592\\
2.03138235221671	0.722403694990678\\
2.11839372920634	0.556434790528121\\
2.19843643687987	0.437651032331631\\
2.27254440903359	0.35040791473491\\
2.34153728052055	0.284894612750398\\
2.40607580165812	0.234745927254051\\
2.46670042347455	0.1957092037518\\
2.5238588373145	0.164869567563888\\
2.57792605858478	0.140183600821197\\
};

\end{axis}
\end{tikzpicture}}
		\caption{Results for the gradiometry problem. Depicted are different noise levels $\sigma$ (x-axis) against empirical mean squared errors $\E{\left\Vert \hat f_\alpha - f\right\Vert_{\X}^2}$ simulated in $M = 10^4$ Monte Carlo  runs, for different parameter choices $\alpha$, namely SOLIT (eq.~\ref{e:limit:a}--\ref{e:limit:d}, \ref{sho_grad_solit}), optimal alpha (eq.~\ref{eq:alpha_opt}, \ref{sho_grad_opt}) and the oracle choice (eq.~\ref{e:orcr}, \ref{sho_grad_oracle}). Furthermore shown is the empirical price of adaptation (eq.~\ref{eq:poa}, \ref{sho_grad_poa}) and a slope (\ref{sho_grad_orc}) indicating the optimal rate of convergence $\mathcal O \left(\left(-\log\sigma\right)^{-3}\right)$.}
		\label{fig:gradiometry}
	\end{figure}
	
	\subsubsection{The heat problem}\label{sec:heat}
	
	As exact solution we choose again
	\[
	f^\dagger \left(x\right) = \frac{\pi}{2} - \left|x\right|, \qquad x \in \left[-\pi,\pi\right].
	\]
	Similarly as in \cite[Remark 15]{h00} it can be seen that the optimal rate of convergence is $\mathcal O \bigl(\left(-\log\sigma\right)^{-3/2+\varepsilon}\bigr)$ for any $\varepsilon>0$. Again, the ideal data $g$ can be computed analytically as a infinite sum, which is in our simulations truncated at a different value that the formula used for the implementation of the operator $T$. We set $\bar t = 0.1$.
	
	The results depicted in Figure \ref{fig:heat} show once more that the rate obtained for SOLIT is very close to the oracle rate. However, it seems unclear if his rate of convergence equals the minimax rate (which might again be due to the difficulty of choosing suitable candidate parameters $\alpha_k$), and the behavior of the empirical price of adaptation especially for spectral cut-off regularization is somewhat irregular.
	
	\begin{figure}[!htb]
		\setlength{\fwidth}{3.1cm}
		\setlength{\fheight}{3.1cm}
		\subfigure[Tikhonov]{\begin{tikzpicture}[baseline]

\begin{axis}[%
width=\fwidth,
height=\fheight,
scale only axis,
xmin=1.2,
xmax=2.7,
xminorticks=false,
xtick={1.52718,1.93264,2.6258},
xticklabels = {$10^{-2}$,$10^{-3}$,$10^{-6}$},
x dir=reverse,
ymode=log,
ymin=0.05,
ymax=10,
yminorticks=false
]
\addplot [color=blue,dashed, line width = 0.35mm]
  table[row sep=crcr]{%
1.24292499185244	0.975478739664182\\
1.42524654864639	0.666611996525687\\
1.57939722847365	0.463156231352703\\
1.71292862109817	0.367588180940498\\
1.83071165675456	0.319260802883624\\
1.93607217241238	0.29040721213431\\
2.03138235221671	0.26641945549389\\
2.11839372920634	0.158564633896099\\
2.19843643687987	0.142891222449722\\
2.27254440903359	0.135026690198909\\
2.34153728052055	0.131106823081512\\
2.40607580165812	0.127907066639165\\
2.46670042347455	0.0955538007469746\\
2.5238588373145	0.0720452346963835\\
2.57792605858478	0.0696553191094004\\
2.62921935297233	0.0682560581667734\\
};

\addplot [color=black!50!green,densely dotted, line width = 0.35mm]
  table[row sep=crcr]{%
1.24292499185244	0.618054638706466\\
1.42524654864639	0.423726391639008\\
1.57939722847365	0.328953184511056\\
1.71292862109817	0.265328915135003\\
1.83071165675456	0.21393566338254\\
1.93607217241238	0.167944678341534\\
2.03138235221671	0.142966250606795\\
2.11839372920634	0.128144437662801\\
2.19843643687987	0.113804923373599\\
2.27254440903359	0.0991549641161567\\
2.34153728052055	0.083561404344954\\
2.40607580165812	0.0735485822849022\\
2.46670042347455	0.0685607239609993\\
2.5238588373145	0.0642069592053288\\
2.57792605858478	0.0593347633981968\\
2.62921935297233	0.054312100190304\\
};

\addplot [color=orange,dash dot, line width = 0.35mm]
  table[row sep=crcr]{%
1.24292499185244	0.65363892235629\\
1.42524654864639	0.443773425697594\\
1.57939722847365	0.353036139293654\\
1.71292862109817	0.307383426941451\\
1.83071165675456	0.237187187506441\\
1.93607217241238	0.173950129488207\\
2.03138235221671	0.148698113950975\\
2.11839372920634	0.137307925703092\\
2.19843643687987	0.131504388893076\\
2.27254440903359	0.127232765917372\\
2.34153728052055	0.0857866769064748\\
2.40607580165812	0.0749867623094092\\
2.46670042347455	0.0705728013909638\\
2.5238588373145	0.0686655721470847\\
2.57792605858478	0.0678190173393523\\
2.62921935297233	0.067271236638863\\
};

\addplot [color=black,dash dot dot, line width = 0.35mm]
  table[row sep=crcr]{%
1.24292499185244	6.44714369388595\\
1.42524654864639	3.72517857836238\\
1.57939722847365	2.18616713838283\\
1.71292862109817	1.37061669997975\\
1.83071165675456	1.82759420003658\\
1.93607217241238	1.82055312378476\\
2.03138235221671	0.923508913392674\\
2.11839372920634	0.558683551122788\\
2.19843643687987	0.434679564207846\\
2.27254440903359	0.313739334260554\\
2.34153728052055	1.32892111153476\\
2.40607580165812	0.653417557868268\\
2.46670042347455	0.312694413660384\\
2.5238588373145	0.209715376087302\\
2.57792605858478	0.154280665623845\\
2.62921935297233	0.121678380436657\\
};

\addplot [color=red,solid, line width = 0.35mm]
  table[row sep=crcr]{%
1.71292862109817	4.16953657126024\\
1.83071165675456	3.49429046985887\\
1.93607217241238	2.98347750656297\\
2.03138235221671	2.58603108033047\\
2.11839372920634	2.26960823632473\\
2.19843643687987	2.01283297888707\\
2.27254440903359	1.80107050611323\\
2.34153728052055	1.62399945558896\\
2.40607580165812	1.47415379197171\\
2.46670042347455	1.34601340738085\\
2.5238588373145	1.23541824333637\\
2.57792605858478	1.1391800388152\\
};

\end{axis}
\end{tikzpicture}}
		\subfigure[Showalter]{\begin{tikzpicture}[baseline]

\begin{axis}[%
width=\fwidth,
height=\fheight,
scale only axis,
xmin=1.2,
xmax=2.7,
xminorticks=false,
xtick={1.52718,1.93264,2.6258},
xticklabels = {$10^{-2}$,$10^{-3}$,$10^{-6}$},
x dir=reverse,
ymode=log,
ymin=0.05,
ymax=10,
yminorticks=false
]
\addplot [color=blue,dashed, line width = 0.35mm]
  table[row sep=crcr]{%
1.24292499185244	0.872018264560948\\
1.42524654864639	0.796331522637542\\
1.57939722847365	0.444574473853662\\
1.71292862109817	0.31583591114682\\
1.83071165675456	0.299787272113757\\
1.93607217241238	0.293936163914749\\
2.03138235221671	0.290747081975436\\
2.11839372920634	0.138114614285917\\
2.19843643687987	0.131605798432062\\
2.27254440903359	0.130380233927232\\
2.34153728052055	0.129931248159934\\
2.40607580165812	0.129880167286254\\
2.46670042347455	0.0966610339862536\\
2.5238588373145	0.0681397274490833\\
2.57792605858478	0.0676348899779918\\
2.62921935297233	0.0675606496919265\\
};
\label{sho_heat_solit}

\addplot [color=black!50!green,densely dotted, line width = 0.35mm]
  table[row sep=crcr]{%
1.24292499185244	0.568030237480298\\
1.42524654864639	0.391234048487936\\
1.57939722847365	0.313895761282014\\
1.71292862109817	0.262561609464909\\
1.83071165675456	0.207700528331047\\
1.93607217241238	0.160276030164206\\
2.03138235221671	0.13817866805714\\
2.11839372920634	0.126035095026467\\
2.19843643687987	0.112995296989676\\
2.27254440903359	0.0974315160477326\\
2.34153728052055	0.0813253390631532\\
2.40607580165812	0.0715901118655842\\
2.46670042347455	0.0675897461280236\\
2.5238588373145	0.0640851518920954\\
2.57792605858478	0.0591232230717065\\
2.62921935297233	0.0538025249595201\\
};
\label{sho_heat_opt}

\addplot [color=orange,dash dot, line width = 0.35mm]
  table[row sep=crcr]{%
1.24292499185244	0.603948475549504\\
1.42524654864639	0.410449575523851\\
1.57939722847365	0.33120190282049\\
1.71292862109817	0.305371725533585\\
1.83071165675456	0.223505168152498\\
1.93607217241238	0.166634559440717\\
2.03138235221671	0.142171255273235\\
2.11839372920634	0.133662083732442\\
2.19843643687987	0.131037969857751\\
2.27254440903359	0.130107819095951\\
2.34153728052055	0.0849358954149372\\
2.40607580165812	0.0731671595082892\\
2.46670042347455	0.0691324187729631\\
2.5238588373145	0.0679493222415927\\
2.57792605858478	0.0676348899779918\\
2.62921935297233	0.0675090220526711\\
};
\label{sho_heat_oracle}

\addplot [color=black,dash dot dot, line width = 0.35mm]
  table[row sep=crcr]{%
1.24292499185244	9.2551556310492\\
1.42524654864639	4.73513254801971\\
1.57939722847365	1.95195807380996\\
1.71292862109817	0.916772926893901\\
1.83071165675456	3.58883659755064\\
1.93607217241238	2.5966373779772\\
2.03138235221671	1.15718514452779\\
2.11839372920634	0.500266172580972\\
2.19843643687987	0.272036959635298\\
2.27254440903359	0.166738015128052\\
2.34153728052055	1.88065647914222\\
2.40607580165812	0.822332945931205\\
2.46670042347455	0.336090638649255\\
2.5238588373145	0.178061067728752\\
2.57792605858478	0.107316630178608\\
2.62921935297233	0.0759136591493278\\
};
\label{sho_heat_poa}

\addplot [color=red,solid, line width = 0.35mm]
  table[row sep=crcr]{%
1.71292862109817	4.16953657126024\\
1.83071165675456	3.49429046985887\\
1.93607217241238	2.98347750656297\\
2.03138235221671	2.58603108033047\\
2.11839372920634	2.26960823632473\\
2.19843643687987	2.01283297888707\\
2.27254440903359	1.80107050611323\\
2.34153728052055	1.62399945558896\\
2.40607580165812	1.47415379197171\\
2.46670042347455	1.34601340738085\\
2.5238588373145	1.23541824333637\\
2.57792605858478	1.1391800388152\\
};
\label{sho_heat_ocr}

\end{axis}
\end{tikzpicture}%
}
		\subfigure[Spectral cut-off]{\begin{tikzpicture}[baseline]

\begin{axis}[%
width=\fwidth,
height=\fheight,
scale only axis,
xmin=1.2,
xmax=2.7,
xminorticks=false,
xtick={1.52718,1.93264,2.6258},
xticklabels = {$10^{-2}$,$10^{-3}$,$10^{-6}$},
x dir=reverse,
ymode=log,
ymin=0.05,
ymax=10,
yminorticks=false
]
\addplot [color=blue,dashed, line width = 0.35mm]
  table[row sep=crcr]{%
1.24292499185244	0.983112617157345\\
1.42524654864639	0.968061812311923\\
1.57939722847365	0.327427871604338\\
1.71292862109817	0.306811510669348\\
1.83071165675456	0.304331247739858\\
1.93607217241238	0.303713549003219\\
2.03138235221671	0.302441201590549\\
2.11839372920634	0.132327702678806\\
2.19843643687987	0.131397089706769\\
2.27254440903359	0.13115887413472\\
2.34153728052055	0.131100080662227\\
2.40607580165812	0.13098964879145\\
2.46670042347455	0.111472265968311\\
2.5238588373145	0.067895042776379\\
2.57792605858478	0.0676535653735625\\
2.62921935297233	0.0675926163742698\\
};

\addplot [color=black!50!green,densely dotted, line width = 0.35mm]
  table[row sep=crcr]{%
1.24292499185244	0.503379721863823\\
1.42524654864639	0.354956313719264\\
1.57939722847365	0.306783996581759\\
1.71292862109817	0.259620164359516\\
1.83071165675456	0.196510933139492\\
1.93607217241238	0.151076965970062\\
2.03138235221671	0.135819828156028\\
2.11839372920634	0.130855653549419\\
2.19843643687987	0.123767763360951\\
2.27254440903359	0.108297112873446\\
2.34153728052055	0.0862333251143964\\
2.40607580165812	0.072842822215338\\
2.46670042347455	0.0685531833076534\\
2.5238588373145	0.0660944892143017\\
2.57792605858478	0.0619029736039033\\
2.62921935297233	0.0549994563421415\\
};

\addplot [color=orange,dash dot, line width = 0.35mm]
  table[row sep=crcr]{%
1.24292499185244	1.03411211194287\\
1.42524654864639	0.984174497183428\\
1.57939722847365	0.971498983243297\\
1.71292862109817	0.96837558227078\\
1.83071165675456	0.314933643728239\\
1.93607217241238	0.306434856966399\\
2.03138235221671	0.304226175736311\\
2.11839372920634	0.30368610922553\\
2.19843643687987	0.131397089706769\\
2.27254440903359	0.13115887413472\\
2.34153728052055	0.106492260888866\\
2.40607580165812	0.0984294496990443\\
2.46670042347455	0.0963513770069548\\
2.5238588373145	0.0958466837337356\\
2.57792605858478	0.0957227803825343\\
2.62921935297233	0.095690698413593\\
};

\addplot [color=black,dash dot dot, line width = 0.35mm]
  table[row sep=crcr]{%
1.24292499185244	3.74829685423676\\
1.42524654864639	1.46611368913078\\
1.57939722847365	0.742019793374747\\
1.71292862109817	0.4404678407079\\
1.83071165675456	1.31926128348421\\
1.93607217241238	0.550281105381859\\
2.03138235221671	0.297132906637755\\
2.11839372920634	0.202567768027601\\
2.19843643687987	0.27183000642975\\
2.27254440903359	0.171287258614216\\
2.34153728052055	2.47313266496472\\
2.40607580165812	0.826965399225392\\
2.46670042347455	0.338221071994888\\
2.5238588373145	0.173929694107031\\
2.57792605858478	0.108098415557683\\
2.62921935297233	0.0753555783420732\\
};

\addplot [color=red,solid, line width = 0.35mm]
  table[row sep=crcr]{%
1.71292862109817	4.16953657126024\\
1.83071165675456	3.49429046985887\\
1.93607217241238	2.98347750656297\\
2.03138235221671	2.58603108033047\\
2.11839372920634	2.26960823632473\\
2.19843643687987	2.01283297888707\\
2.27254440903359	1.80107050611323\\
2.34153728052055	1.62399945558896\\
2.40607580165812	1.47415379197171\\
2.46670042347455	1.34601340738085\\
2.5238588373145	1.23541824333637\\
2.57792605858478	1.1391800388152\\
};

\end{axis}
\end{tikzpicture}%
}
		\caption{Results for the heat problem. Depicted are different noise levels $\sigma$ (x-axis) against empirical mean squared errors $\E{\left\Vert \hat f_\alpha - f\right\Vert_{\X}^2}$ simulated in $M = 10^4$ Monte Carlo  runs, for different parameter choices $\alpha$, namely SOLIT (eq.~\ref{e:limit:a}--\ref{e:limit:d}, \ref{sho_heat_solit}), optimal alpha (eq.~\ref{eq:alpha_opt}, \ref{sho_heat_opt}) and the oracle choice (eq.~\ref{e:orcr}, \ref{sho_heat_oracle}). Furthermore shown is the empirical price of adaptation (eq.~\ref{eq:poa}, \ref{sho_heat_poa}) and a slope (\ref{sho_heat_ocr}) indicating the optimal rate of convergence $\mathcal O \left(\left(-\log\sigma\right)^{-\frac32}\right)$.}
		\label{fig:heat}
	\end{figure}
	
	\subsubsection{Comparison with the classical Lepski{\u\i}-type balancing principle}
	
	{\rev For the sake of completeness, we also compare the performance of SOLIT with the classical Lepski{\u\i}-type balancing principle \eref{eq:lepskij} for the Tikhonov regularization. For a theoretical analysis of this parameter choice method \eref{eq:lepskij} we refer to \cite{mp03,bh05,m06,MP06,wh12} and emphasize that in all these papers the tuning parameter $\kappa \geq 1$ has to be chosen in dependence of the noise level $\sigma$. However, in our simulations we fix $\kappa = 1$ as in \cite{w18}, so that the corresponding Lepski{\u\i} method will empirically perform well though without a theoretical justification, see also \cite{bl11}. }
	
	\begin{figure}[!htb]
		\setlength{\fwidth}{3.1cm}
		\setlength{\fheight}{3.1cm}
		\subfigure[Antiderivative]{\begin{tikzpicture}[baseline]

\begin{axis}[%
width=\fwidth,
height=\fheight,
scale only axis,
xmode=log,
xmin=1e-08,
xmax=0.03125,
xminorticks=false,
xtick={0.00000001,0.000001,0.0001,0.01},
xticklabels = {$10^{-8}$,$10^{-4}$,$10^{-2}$},
ymode=log,
ymin=1e-05,
ymax=100,
yminorticks=false
]
\addplot [color=blue,dashed, line width = 0.35mm]
   table[row sep=crcr]{%
0.03125	3.14591653791473\\
0.015625	1.58605553480818\\
0.0078125	0.967023962122085\\
0.00390625	0.618600255278018\\
0.001953125	0.458512534085527\\
0.0009765625	0.292612226153668\\
0.00048828125	0.152770031696011\\
0.000244140625	0.0839076830003627\\
0.0001220703125	0.0485157463442191\\
6.103515625e-05	0.0278556619582846\\
3.0517578125e-05	0.0159053951388464\\
1.52587890625e-05	0.00913939272884465\\
7.62939453125e-06	0.00523082709257938\\
3.814697265625e-06	0.00300860942750002\\
1.9073486328125e-06	0.00173005269205044\\
9.5367431640625e-07	0.000999466742315228\\
4.76837158203125e-07	0.000577524349520822\\
2.38418579101562e-07	0.000335315726614192\\
1.19209289550781e-07	0.00019438855559016\\
5.96046447753906e-08	0.000113137759892272\\
2.98023223876953e-08	6.5747835080111e-05\\
};

\addplot [color=green, dotted, line width = 0.35mm]
  table[row sep=crcr]{%
0.03125	8.29785921278785\\
0.015625	3.63343880818637\\
0.0078125	2.02740113566384\\
0.00390625	1.18343771869766\\
0.001953125	0.736103488990482\\
0.0009765625	0.490777402815259\\
0.00048828125	0.300349736915832\\
0.000244140625	0.168432705613506\\
0.0001220703125	0.1008788095177\\
6.103515625e-05	0.0612324757398027\\
3.0517578125e-05	0.036938617779984\\
1.52587890625e-05	0.0214022467431717\\
7.62939453125e-06	0.0122903212939352\\
3.814697265625e-06	0.00745515097549409\\
1.9073486328125e-06	0.00455048408185969\\
9.5367431640625e-07	0.00277197789181621\\
4.76837158203125e-07	0.00160322978084378\\
2.38418579101562e-07	0.000905551755873286\\
1.19209289550781e-07	0.00055087968009488\\
5.96046447753906e-08	0.000335847261715766\\
2.98023223876953e-08	0.000204482559521951\\
};

\addplot [color=red,solid, line width = 0.35mm]
  table[row sep=crcr]{%
9.5367431640625e-07	3.14591653791473\\
4.76837158203125e-07	1.87057316504646\\
2.38418579101562e-07	1.1122494585032\\
1.19209289550781e-07	0.661347484854967\\
5.96046447753906e-08	0.393239567239341\\
};

\end{axis}
\end{tikzpicture}%
}
		\subfigure[Gradiometry]{\begin{tikzpicture}[baseline]

\begin{axis}[%
width=\fwidth,
height=\fheight,
scale only axis,
xmin=1.2,
xmax=3,
xminorticks=false,
xtick={1.52718,2.22033,2.91347},
xticklabels = {$10^{-2}$,$10^{-4}$,$10^{-8}$},
x dir=reverse,
ymode=log,
ymin=0.005,
ymax=10,
yminorticks=false
]
\addplot [color=blue,dashed, line width = 0.35mm]
  table[row sep=crcr]{%
1.24292499185244	2.87909523346791\\
1.42524654864639	1.16777442535756\\
1.57939722847365	0.479504726190234\\
1.71292862109817	0.246149492684402\\
1.83071165675456	0.116566977892671\\
1.93607217241238	0.0777382056021961\\
2.03138235221671	0.0530964431324207\\
2.11839372920634	0.0414231563444821\\
2.19843643687987	0.0341081045023091\\
2.27254440903359	0.0275288068984984\\
2.34153728052055	0.0225781756783927\\
2.40607580165812	0.0184206330225713\\
2.46670042347455	0.0153642557373841\\
2.5238588373145	0.0127967978584949\\
2.57792605858478	0.0109014010531315\\
2.62921935297233	0.00923411854278051\\
2.67800951714176	0.00799087583786051\\
2.72452953277665	0.00687811498022994\\
2.76898129534749	0.00602202832464446\\
2.81154090976628	0.00525038238350084\\
2.85236290428654	0.00464916864149194\\
};

\addplot [color=green, dotted, line width = 0.35mm]
  table[row sep=crcr]{%
1.24292499185244	16.0435337135039\\
1.42524654864639	5.5466857421135\\
1.57939722847365	2.304257462614\\
1.71292862109817	0.989585737963332\\
1.83071165675456	0.437931520477081\\
1.93607217241238	0.209593872796117\\
2.03138235221671	0.114103625958202\\
2.11839372920634	0.0641459194851208\\
2.19843643687987	0.0487098540729931\\
2.27254440903359	0.0367626324095552\\
2.34153728052055	0.0291672115639124\\
2.40607580165812	0.0234820213731387\\
2.46670042347455	0.0190293682880547\\
2.5238588373145	0.01578255650663\\
2.57792605858478	0.0131294610225409\\
2.62921935297233	0.0111084979777438\\
2.67800951714176	0.00940936436587437\\
2.72452953277665	0.00811071554294788\\
2.76898129534749	0.00697364617351146\\
2.81154090976628	0.00609553777095486\\
2.85236290428654	0.00530786966092076\\
};

\addplot [color=red,solid, line width = 0.35mm]
  table[row sep=crcr]{%
1.71292862109817	1.87796741803241\\
1.83071165675456	1.3189565405111\\
1.93607217241238	0.961519318032592\\
2.03138235221671	0.722403694990678\\
2.11839372920634	0.556434790528121\\
2.19843643687987	0.437651032331631\\
2.27254440903359	0.35040791473491\\
2.34153728052055	0.284894612750398\\
2.40607580165812	0.234745927254051\\
2.46670042347455	0.1957092037518\\
2.5238588373145	0.164869567563888\\
2.57792605858478	0.140183600821197\\
};

\end{axis}
\end{tikzpicture}%
}
		\subfigure[Heat]{\begin{tikzpicture}[baseline]

\begin{axis}[%
width=\fwidth,
height=\fheight,
scale only axis,
xmin=1.2,
xmax=3,
xminorticks=false,
xtick={1.52718,2.22033,2.91347},
xticklabels = {$10^{-2}$,$10^{-4}$,$10^{-8}$},
x dir=reverse,
ymode=log,
ymin=0.05,
ymax=10,
yminorticks=false
]

\addplot [color=blue,dashed, line width = 0.35mm]
   table[row sep=crcr]{%
1.24292499185244	0.972692773696537\\
1.42524654864639	0.668228837723599\\
1.57939722847365	0.463687776373152\\
1.71292862109817	0.367291818226943\\
1.83071165675456	0.319110421666781\\
1.93607217241238	0.290277071372924\\
2.03138235221671	0.26643492694494\\
2.11839372920634	0.158671108031446\\
2.19843643687987	0.142919623463261\\
2.27254440903359	0.135018786703197\\
2.34153728052055	0.131110236537239\\
2.40607580165812	0.127911974690013\\
2.46670042347455	0.0952485258552639\\
2.5238588373145	0.0720531989792376\\
2.57792605858478	0.0696660530749462\\
2.62921935297233	0.06825562363946\\
2.67800951714176	0.0676674884996674\\
2.72452953277665	0.0673264145519381\\
2.76898129534749	0.0662331980711002\\
2.81154090976628	0.04121924274665\\
2.85236290428654	0.04015326902892\\
};

\addplot [color=green,dotted, line width = 0.35mm]
  table[row sep=crcr]{%
1.24292499185244	2.74489709710471\\
1.42524654864639	1.51509618046226\\
1.57939722847365	0.867995556854201\\
1.71292862109817	0.547558820073195\\
1.83071165675456	0.400576698684373\\
1.93607217241238	0.333165138693539\\
2.03138235221671	0.292848445777196\\
2.11839372920634	0.209071248796984\\
2.19843643687987	0.16187986373989\\
2.27254440903359	0.142118667745703\\
2.34153728052055	0.135455277638012\\
2.40607580165812	0.129845329543719\\
2.46670042347455	0.101012528543732\\
2.5238588373145	0.080551968170139\\
2.57792605858478	0.0728340182598924\\
2.62921935297233	0.0696024875946655\\
2.67800951714176	0.0683742119355562\\
2.72452953277665	0.0676434718159922\\
2.76898129534749	0.0583339364239753\\
2.81154090976628	0.0461634039018943\\
2.85236290428654	0.041957785880249\\
};
\label{lep}

\addplot [color=red,solid, line width = 0.35mm]
  table[row sep=crcr]{%
1.71292862109817	4.16953657126024\\
1.83071165675456	3.49429046985887\\
1.93607217241238	2.98347750656297\\
2.03138235221671	2.58603108033047\\
2.11839372920634	2.26960823632473\\
2.19843643687987	2.01283297888707\\
2.27254440903359	1.80107050611323\\
2.34153728052055	1.62399945558896\\
2.40607580165812	1.47415379197171\\
2.46670042347455	1.34601340738085\\
2.5238588373145	1.23541824333637\\
2.57792605858478	1.1391800388152\\
};

\end{axis}
\end{tikzpicture}%
}
		\caption{{\rev Comparison of SOLIT with the classical Lepski{\u\i}-type balancing principle \eref{eq:lepskij} for all three different testing problems with Tikhonov regularization. Depicted are different noise levels $\sigma$ (x-axis) against empirical mean squared errors $\E{\left\Vert \hat f_\alpha - f\right\Vert_{\X}^2}$ simulated in $M = 10^4$ Monte Carlo  runs, for different parameter choices $\alpha$, namely, SOLIT (eq.~\ref{e:limit:a}--\ref{e:limit:d}, \ref{sho_heat_solit}), Lepski{\u\i} (eq.~\ref{eq:lepskij} with $\kappa = 1$, \ref{lep}) and the optimal rate of convergence (\ref{sho_heat_ocr}).}}
		\label{fig:lepskij}
	\end{figure}
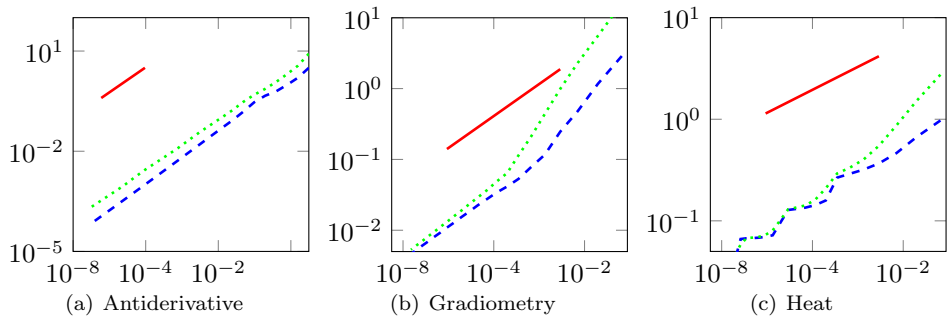
	
{\rev The results depicted in Figure \ref{fig:lepskij} show a comparable behavior for all three testing problems with a slightly superior performance of SOLIT. However, it cannot be expected that the potential difference of a $\log$-factor will be visible in such a simulation. Beyond these numerical findings, we emphasize that SOLIT is proven to yield the optimal rate of convergence in many situations without a need of tuning parameters in dependency of the noise level $\sigma$ and the ill-posedness of the problem, which is a substantial advantage over the classical Lepski{\u\i}-type balancing principle.}

\section{Conclusion and outlook}\label{sec:conclusions}
	
	In the setup of statistical linear inverse problems with white (i.e., independent) Gaussian noises, we have investigated the SOLIT rule of choosing the regularization parameter for (ordered) filter based regularization methods. SOLIT is an automatic and data adaptive procedure, as it requires only the knowledge of forward operator and the noise level. From a computational perspective, it involves $\mathcal O(\log \sigma)$ number of computations of the estimator $\hat f_{\alpha}$ and $\mathcal O\bigl((\log \sigma)^2\bigr)$ number of evaluations of the variance estimate $\trace \bigl(\svar{\hat f_{\alpha}}\bigr)$. The overall computational complexity of SOLIT is thus of nearly (i.e., up to log-factors) the same order as that of evaluation of  $\hat f_{\alpha}$ for a fixed~$\alpha$. In particular, a singular value decomposition of the forward operator is not required, unless spectral cut-off and its variants are used as regularization methods. In short, SOLIT is piratically attractive for various applications (even with massive data), which has also been demonstrated by our simulation study in \cref{sec:implementation}. 
	
	Besides desirable empirical performance, the SOLIT rule is shown to attain the sharp order of optimal adaptation rates under various smoothness conditions (including logarithmic and polynomial smoothness) for mildly ill-posed problems. This is in sharp contrast to the standard Lepski\u{\i} rule, where one often losses unnecessarily a log-factor in the convergence rates. In case of severely ill-posed problems, we have also established the convergence rates for SOLIT under general smoothness conditions, but it remains open whether such rates are adaptively optimal or not. 
	
	In addition, we would like to point out that there are at least two ways of estimation of the noise level (required by SOLIT) when it is not available in practice. The one approach is to employ bootstrap procedures as it is done in \cite{SpWi19}, which allows to estimate the threshold $\thd_{m_1, m_2}$ (cf.\ \cref{d:solit}) without knowing the noise level. The other approach is to estimate the noise level directly from the data, see, e.g., \cite{Spo02,MBWF05,ArvL12}. Further, it may be possible to extend SOLIT beyond the Gaussian noises. If the distribution of the noises is known, e.g., Poisson, one could modify SOLIT by simply replacing $z_{m_1,m_2}(x)$ (cf.\ \cref{d:solit}) with the $(1-x)$-quantile of the noise distribution. Otherwise, one could attempt to estimate such a quantile from the data, e.g., by following the idea in \cite{ZoYu08}. All of these extensions of SOLIT are interesting topics for future research.

\section*{Acknowledgments}
	
	FW wants to thank M. Reiss and V. Spokoiny for pointing his attention to the topic at hand and especially the inspiring paper \cite{SpWi19} while he was visiting HU Berlin and WIAS. HL is funded by the Deutsche Forschungsgemeinschaft (DFG, German Research Foundation) under Germany’s Excellence Strategy--EXC 2067/1-390729940, and acknowledges the support of DFG Collaborative Research Center 1456.
	
		\appendix
	
	\section{Adaptation rates}\label{s:lbar}

	We show a simple and interesting fact that the ordered filters with the posteriori parameter choice $\alpha \asymp \sigma$ attain the sharp order of optimal adaptation rates for logarithmic smoothness classes. {\rev This is because the upper bound on MSE in \cref{p:nlr} below coincides (up to multiplying constants) with the well-known lower bound in this smoothness class (see e.g.~\cite{Pin80})}. 
	
	\begin{prop}[Noise-level-rule]\label{p:nlr}
		Assume the model \eref{e:model} and that the eigenvalues of $\fop^*\fop$ decays as in \eref{e:eigopb} with either $a >1$ and $b = 0$, or $a \in \R$ and $b, \vartheta > 0$. Let Assumptions~\ref{a:scc}, \ref{a:qlf} and \ref{a:ss} hold, where the index function $\varphi = \varphi_{\nu,\tau}$ is given in \eref{e:indfun} with the parameters $\nu = 0$ and $0 < \tau < \tau_0\le \infty$. Let also $\hat\sig_\alpha = \flt_\alpha(\fop^*\fop) \fop^* \data
		$ with $\flt_\alpha(\cdot)$ be an ordered filter. Then, for the noise-level-rule $\alpha_* \asymp \sigma$, it holds that
		$$
		\limsup_{\sigma^2 \to 0} \sup_{0 < \tau < \tau_0 }\; \frac{1}{\psi_{0,\tau}(\sigma^2)} \sup_{\sigdag \in \mathcal W_{\varphi_{0,\tau}}\left(\rho\right)} \mathbb E \left[\bigl\Vert \hat \sig_{\alpha_*} - \sigdag\bigr\Vert_{\X}^2\right] < \infty,
		$$
		where $\psi_{0,\tau}(\sigma^2) := \bigl(-\log (\sigma^2)\bigr)^{-2\tau}$.
	\end{prop}
	
	\begin{proof}
		Consider an arbitrary $\tau \in (0, \tau_0)$. By \cref{l:bvd} we obtain
		\begin{equation}\label{e:bvnlr}
			\sup_{\sigdag \in  \mathcal{W}_{\varphi_{0,\tau}} (\rho)} \E{\norm{\hat \sig_{\alpha_*} - \sigdag}^2} \lesssim  \varphi_{0,\tau}(\sigma)^2 + \sigma^2 \frac{S(\sigma)}{\sigma}.     
		\end{equation}
		In case of mild ill-posedness, namely, $a > 1$ and $b = 0$, we have the surrogate function $S(x) \asymp x^{-1/a}$, which together with \eref{e:bvnlr} implies 
		$$
		\sup_{\sigdag \in  \mathcal{W}_{\varphi_{0,\tau}} (\rho)} \E{\norm{\hat \sig_{\alpha_*} - \sigdag}^2} \lesssim \bigl( -\log(\sigma^2)\bigr)^{-2\tau} + \sigma^{1- 1/a} \lesssim \bigl( -\log(\sigma^2)\bigr)^{-2\tau}.
		$$
		In case of severe ill-posedness, namely, $a \in \R$ and $b, \vartheta > 0$, we have the surrogate function $S(x) \asymp (-\log x)^{1/\vartheta}$, which together with \eref{e:bvnlr} implies 
		$$
		\sup_{\sigdag \in  \mathcal{W}_{\varphi_{0,\tau}} (\rho)} \E{\norm{\hat \sig_{\alpha_*} - \sigdag}^2} \lesssim \bigl( -\log(\sigma^2)\bigr)^{-2\tau} + \sigma (-\log \sigma)^{1/\vartheta} \lesssim \bigl( -\log(\sigma^2)\bigr)^{-2\tau}.
		$$
		Combining two cases concludes the proof. 
	\end{proof}
	

\end{document}